\documentclass[12pt]{amsart}
\usepackage{graphicx}
\usepackage{fullpage}
\usepackage{hyperref}
\usepackage{amsfonts}
\usepackage{amsthm}
\usepackage{amsmath}
\usepackage{float}
\usepackage{listings}
\usepackage{overpic}
\usepackage{xcolor}
\usepackage{latexsym,amssymb}
\usepackage[utf8]{inputenc}
\usepackage[margin=0.9in]{geometry}
\numberwithin{equation}{section}
\theoremstyle{plain}
\usepackage{tikz}
\usepackage{dsfont}
\usepackage{xypic}
\usepackage{caption,subcaption}
\usepackage{epsfig,graphicx,graphics,color,tikz,tikz-cd,tcolorbox}
\usepackage{enumerate}
\usepackage[sort,noadjust]{cite}
\usepackage{todonotes}
\usepackage{newtxmath}
\newtheorem{theorem}{Theorem}[section]
\newtheorem{prop}[theorem]{Proposition}
\newtheorem{assume}{Assumption}
\newtheorem{lemma}[theorem]{Lemma}

\newtheorem{remark}[theorem]{Remark}
\newtheorem{defn}{Definition}
\newtheorem{example}[theorem]{Example}
\newcommand{\vol}{{\rm vol }}
\newcommand{\R}{{\mathbb R}}

\newcommand{\N}{{\mathbb N}}

\newcommand{\B}{\ensuremath{\mathcal{B}}}

\DeclareMathOperator{\dist}{dist}

\title{Graphon as a Bridge between Graphs and Manifolds}

\author{Dong Zhang}
\address{School of Mathematical Sciences,  Peking University,   100871 Beijing, China.}
\email{dongzhang@math.pku.edu.cn}

\begin{document}


\normalsize
\allowdisplaybreaks[4]



%
\begin{abstract}
We show that there exist graphons that interpolate between Riemannian manifolds and weighted geometric graphs. Specifically, the  graph-to-manifold approximation used in manifold learning can be regarded as the composition of a graph-to-graphon convergence and a graphon-to-manifold convergence in a certain sense.  
Furthermore, we establish a monotonicity inequality which reveals an implicit relationship between numerous combinatorial parameters and geometric quantities on graphons. Using this inequality, we find relations among conductance, maxcut problem, capacity, and  packing radius, as well as their limiting behaviors under graph-to-graphon and graphon-to-manifold convergences; some of these relations are novel even for simple graphs and closed manifolds. 
We also find that in the graphon-to-manifold limiting process, the purely combinatorial parameters diverge, leading to a ``blowup'' phenomenon; and continuous versions of some graph parameters cannot be realized via Riemannian manifolds, but can be realized via graphons.

\vspace{0.3cm}

\end{abstract}

\maketitle

\tableofcontents


\vspace{0.5cm}

\section{Introduction and Preliminary}

Weighted geometric graphs have been extensively studied in the fields of data science, manifold learning, and numerical analysis. In particular, such geometric graphs frequently appear in the graph-to-manifold approximation used in manifold learning \cite{Belkin-Niyogi,XuSinger26}. 
Research closely related to this topic includes, at the very least: extensive studies on random geometric graphs \cite{Penrose}, and systematic research on the continuum limit for the spectrum of the graph Laplacian \cite{Belkin-Niyogi,GarciaSlepcev18}. 
 
In a quite different research area, graphons arise both as a natural notion for the limit of a sequence of dense graphs, and as the fundamental defining objects of exchangeable random graph models \cite{BCLSV,Lovasz-Szegedy,Bhattacharya/Chatterjee/Janson,GrebikPikhurko}. The theory of graphon has been widely used in statistical physics, random graphs, multiway cuts, community detection, spectral clustering, and learning theory \cite{BCLSV,Backhausz,Backhausz25,Azizpour,chatterjeeSohamBhaswar}. 

The core of this paper lies in interpreting graphon as a bridge connecting graphs and manifolds, or more precisely, in endowing graphon with the significance of linking combinatorics with analysis and geometry. 
In fact, if one regards graphon as an object lying between graphs and manifolds, it appears to strike a balance between combinatorics and geometry, possessing both discrete and analytic properties. The results in this paper include: 
\begin{enumerate}[(1)]
\item 
We construct and analyze implicit graphons corresponding to sampling graphs on manifolds and show a rigorous proof of the graphon-to-manifold convergence, see Theorems 
\ref{thm:pers1-diag}, \ref{th:W_r-M} and \ref{thm:metric-converge}. 
\item 
We propose generalized Sobolev-type constants on graphons, and use these to give a unified study for maxcut, conductance, $p$-capacity and sphere packing radius on graphons; and we further establish a monotonicity inequality linking these combinatorial parameters and geometric quantities on graphons, see Theorems \ref{pro:description} and \ref{thm:monotonicity}. 
\item All the quantities mentioned above are preserved under the graph-to-graphon approximation, and we establish the convergence of Sobolev-type geometric quantities under the graphon-to-manifold approximation. Moreover, in the graphon-to-manifold limiting process, the purely combinatorial parameters do not converge, leading to a ``blowup'' phenomenon, see Theorems \ref{thm:geo-Sobolev-conve} and \ref{thm:Cheeger-set}, Proposition \ref{prop:sphere-radius-converge}, 
Remarks  \ref {rem:max-min-notconverge}, \ref{rem:maxcut-blowup} and \ref{rem:signed-no-manifold-version};  and continuous versions of some graph parameters cannot be realized via Riemannian manifolds, but can be realized via graphons. 

\item 
Under graphon-to-manifold approximation, we prove that Cheeger sets on graphons have  subsequence converging to a Cheeger set on Euclidean domain of finite perimeter; see Theorem \ref{thm:Cheeger-set}. 
\item Using  monotonicity inequalities we established, some important nonlinear eigenvalues, including tensor-like eigenvalues and $p$-Laplacian type eigenvalues, are proved to be bounded by linear eigenvalues of graphons, see Remark \ref{rem:tensor-like}. 
\end{enumerate}
Before going into details about the main content which will be presented in Sections \ref{sec:hidden-limit} and \ref{sec:geometry-combi}, we first present an introduction for cut distance on graphons, $TL^p$ framework and Gamma-convergence. 
\subsection{
Cut distance}

Given a Borel measure space $(J,\mu)$ with finite measure, a generalized \emph{graphon} is a bounded, symmetric, measurable kernel function $W: J\times J \to \mathbb{R}$, and its cut norm is defined as
\[
\|W\|_{\square}: = \sup_{\text{measurable }S, T \subseteq J} \left| \int_{S \times T} W(x,y) \, d\mu(x) \, d\mu(y) \right|.
\]
If such $W$ is further assumed to be nonnegative (i.e., $W(x,y)\ge0$, $\forall (x,y)\in J\times J$), then we call $W$ a graphon. 
For a measure-preserving bijections $\varphi: J \to J$, define
\[
W^{\varphi}(x,y): = W(\varphi(x), \varphi(y)).
\] 
The cut distance between two (generalized) graphons $U$ and $W$ is defined as
\[
\delta_{\square}(U, W): = \inf_{\varphi: J \to J \text{ measure-preserving bijection}} \|U - W^{\varphi}\|_{\square}.
\]
We say that a sequence of graphons $\{W_n\}_{n\ge1}$ converges to $W$ in cut distance if $\delta_{\square}(W_n, W)\to0$ as $n\to\infty$. 
In the literature, $J$ is usually set as the standard interval $[0,1]$, and a graph $G$ is commonly identified with a standard stepgraphon in graph limit theory. 
Hence, the cut distance is convenient for describing the convergence from graphs to graphons; see Appendix \ref{app:cut-distance} for more details. In addition, it would be useful to view graphs as graphons, and most of graph parameters can be set up on graphons as well. 
However, convergence in cut distance is actually a kind of structural convergence, the strength of which is insufficient to guarantee the convergence of graph parameters. 
Hence, in this paper, we not only consider structural convergence involving cut distance, but also consider more on variational convergence involving Gamma-convergence and $TL^p$ convergence.

\subsection{
$TL^p$ convergence}
To describe a limiting process from graphs to manifolds, and in addition, to study the limiting behavior about certain quantities on graphs to quantities on Riemannian manifolds, the $TL^p$ framework would be useful. 
Consider measure spaces $(J_n,\mu_n)$ and $(J,\mu)$. Suppose that $\mathrm{dist}$ is a metric on $\cup_{n\ge1}J_n\cup J$. We say that $f_n\in L^p(J_n,\mu_n)$ converges in $TL^p$-sense to a function $f\in L^p(J,\mu)$ if $d_{TL^p}^p((\mu,f), (\mu_n,f_n))\to0$ as $n\to\infty$, where 
\[ d_{TL^p}^p((\mu,f), (\mu_n,f_n)) := \inf  \Big\{\int_{J\times J_n} \big(\mathrm{dist}(x,y)\big)^p + |f(x) - f_n(y)|^p \, \mathrm{d}\pi(x,y) \, : \, \pi\in\Pi(\mu,\mu_n)  \Big\}\] and $\Pi(\mu,\mu_n) $ denotes the set of transportation plans (or couplings) between $\mu$ and $\mu_n$, that is, the set of measures on the product space $J\times J_n$ whose first marginal is $\mu$ and second marginal is $\mu_n$.  
Hence, with a suitable metric on the union space $\cup_{n\ge1}J_n\cup J$, we can define the convergence of $L^p$-integrable functions on $J_n$ to $L^p$-integrable functions on $J$;  see Appendix \ref{Appsec:TLp} for details. This would be convenient when $J$ is a Riemannian manifold, and $J_n$ is a subset consisting of sampling points in $J$. 
\subsection{
Gamma convergence}
Since we are going to consider the convergence of variational parameters/quantities on graphs/graphons to that of Riemannian manifolds, and most of these quantities are defined as the minimizations or maximizations of certain functionals on $L^p(J,\mu)$, we need a concept of convergence in variational analysis, namely, the Gamma-convergence, to ensure the convergence of these quantities. 
To match our setting, we slightly extend the definition of Gamma-convergence to ensure that it is applicable for more general functional sequences. 
Let $X$ be a function space on $J$, in which we usually set $X$ as a linear subspace of $L^p(J,\mu)$, e.g., $X=L^p(J,\mu)$. Similarly, one can consider a function space $X_n$ on $J_n$. A sequence of functionals $\mathcal{R}_n:X_n\to[0, \infty]$ \emph{Gamma-converges} to the functional $\mathcal{R}: X \to [0, \infty]$ as $n \to \infty$ if there holds the Liminf Inequality $\liminf_{n \to \infty} \mathcal{R}_n(f_n) \ge \mathcal{R}(f)$ for any $f \in X$ and every sequence $\{f_n\}_{n \in \mathbb{N}}$ with $f_n \in X_n$ converging to $f$; and the recovery condition that for every $f \in X$ there exists a sequence $\{f_n\}_{n \in \mathbb{N}}$ with $f_n \in X_n$ converging to $f$ such that
$\lim_{n \to \infty} \mathcal{R}_n(f_n) = \mathcal{R}(f)$. 
Under appropriate assumption such as asymptotic equicoercivity, we have $\inf_{f\in X_n}\mathcal{R}_n(f)\to \inf_{f\in X}\mathcal{R}(f)$ as $n\to\infty$, and 
\begin{equation}\label{eq:min-max-converge}
\lim_{n\to\infty}\inf_{S\subset X_n:\mathrm{genus}(S)\ge k}\sup_{f\in S}\mathcal{R}_n(f)=\inf_{S\subset X:\mathrm{genus}(S)\ge k}\sup_{f\in S}\mathcal{R}(f)    
\end{equation}
where $\mathrm{genus}(S)$ is the Krasnoselskii genus of $S$; see Appendix \ref{sec:generalized-Gamma} for details. Since many geometric-like quantities on $J_n$ and $J$ can be expressed in terms of min-max form, their convergences follow from \eqref{eq:min-max-converge}.

\section*{Acknowledgements} 
Dong Zhang is supported by grants from  the  National Natural Science Foundation of China (No.\ 12401443). The initial idea for this paper comes from the author’s 2017 PhD thesis \cite{zhang/phd}; in particular, Appendix \ref{sec:generalized-Gamma} and the proof of Theorem \ref{thm:Cheeger-set} are English translations of a section from that thesis. Whilst undertaking his postdoctoral research from 2019 to 2022, the author came across  Lov\'asz's limits theory, which led him to conclude that graphon should serve as the bridge linking graphs and manifolds, thereby refining the preliminary ideas he had developed during his doctoral studies.  Moreover, in 2025, the author came across a series of works by Garc\'ia Trillos, Slep\v{c}ev and Thorpe on $TL^p$ convergence and Gamma convergence; this paper was technically inspired by these works. 
Therefore, the author is deeply grateful to Garc\'ia Trillos, Lov\'asz, Slep\v{c}ev and Thorpe for their nice works, which provided crucial inspiration for the development of this paper. 

\vspace{0.1cm}

\noindent\textbf{AI statement}. No AI tools were used in the research process and preparation of this manuscript. 

\section{Hidden graph limits between graphs and manifolds
}
\label{sec:hidden-limit}

In this section, we show that for convergent weighted geometric sampling graph sequences from an underlying manifold, there exist graphons that interpolate in the graph-to-manifold limiting process. 
Therefore, we can regard graphon as an intermediate state in the transition from graphs to manifolds. 
We shall first propose a reasonable and comprehensive set of assumptions. 

\begin{assume}\label{ass:1}
Given a smooth, compact, closed Riemannian manifold $M$ of dimension $k\ge1$, a threshold $r>0$, and finite number of points $V_n:=\{x_1,\cdots,x_n\}\subset M$, we  equip  each pair of vertices $\{x_i,x_j\}$ with the edge weight 
$w_{ij}(r)=K_r(\mathrm{dist}(x_i,x_j))$, where $\mathrm{dist}(x_i,x_j)$ stands for the shortest geodesic  distance between points $x_i$ and $x_j$, $$K_r(t):=\frac{1}{r^k}K\big(\frac{t}{r}\big),\;\forall t\ge 0,$$  
and  $K:[0,+\infty)\to [0,+\infty)$  is a function satisfying: 
$$\begin{cases}
K(0)>0\\
K\text{ is non-increasing}\\
0<\int_0^\infty K(t)t^ndt<+\infty,\;\forall n=1,2,\cdots.
\end{cases}$$
\end{assume}
For convenience, in this paper we usually take $K(t)=e^{-t^2}$ for any $t\ge0$, or take $K(\cdot)$ as a non-increasing function with nonempty compact support. 

Using the function $K$, we define the following constants that will be used frequently in this paper: 
\begin{equation}\label{eq:C_0k,C_pk}
C_{0,k}=\int_{\R^k}K(|x|)dx\;\text{ and }\; C_{p,k}=\int_{\R^k}|x_1|^pK(|x|)dx    \text{ for }1\le p<\infty
\end{equation}
where $k
$ is the dimension of $M$, $|x|$ is the Euclidean norm of $x$, and $x_1$ denotes the first component of $x$. 

The above sampling process in Assumption \ref{ass:1} is widely adopted. %
In machine learning theory, Belkin and Niyogi \cite{Belkin-Niyogi}  proved convergence of empirical graph Laplacians on Riemannian manifolds from uniformly sampled point-cloud data towards the Laplace-Beltrami operator when we make sampling graphs finer and finer. %

\begin{assume}\label{ass:2}
Under the hypotheses of Assumption \ref{ass:1},  
we 
further assume that the sampling 
$V_n:=\{x_1,\cdots,x_n\}\subset M$ is uniformly distributed, that is, for any open ball $B(x,r)$ in $M$,  $$\lim\limits_{n\to+\infty}\frac{\#\big(B(x,r)\cap V_n\big)}{\# V_n}=\frac{\vol(B(x,r))}{\vol(M)} . $$
\end{assume}
We denote the finite weighted graph appearing in the two assumptions above as $$G_{n,r}=(V_n,w(r))$$ where the edge weight function $w(r)$ restricted on $\{x_i,x_j\}$ will be written as $w_{ij}(r):=K_r(\mathrm{dist}(x_i,x_j))$. 
To help readers understand our results, we shall begin with a preliminary general overview. 

\vspace{0.1cm}

 \noindent {\sl\textbf{General Overview}.  
Under Assumptions \ref{ass:1} and \ref{ass:2}, for fixed $r>0$, $\{G_{n,r}\}_{n\ge 1}$ converges, in a certain sense, to a graphon, denoted by $W_r$. And $W_r$ converges to the Riemannian manifold $M$ in a certain sense, as $r$ tends to $0$.  

}

\vspace{0.1cm}

The two limiting processes appearing in the above general overview will be explicitly stated in Sections \ref{sec:graph-graphon} and \ref{sec:graphon-manifold}
.   We emphasize that Theorems \ref{thm:manifold-sampling-1}, \ref{thm:pers1-diag}, \ref{th:W_r-M}, \ref{thm:metric-converge} and \ref{th:gamma-function} 
constitute the detailed formal version of the above general overview.  %

Here, we may simply summarize the above general overview as:
$$
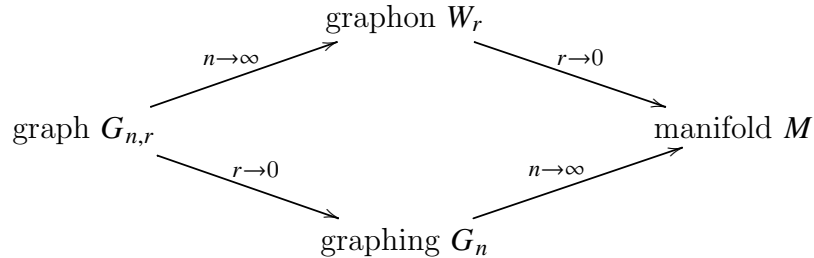
\xymatrix{& & \text{graphon }W_r\ar[rrd]^{r\to0} & & \\  \text{graph }G_{n,r}\ar[rru]^{n\to\infty} & &  & & \text{manifold }M }$$
With a slight abuse of notation on limits, we can also rewrite the above general overview as
$$\lim\limits_{r\to0}\lim\limits_{n\to\infty}G_{n,r}=\lim\limits_{r\to0}W_r=M.$$

This result, together with Proposition \ref{pro:description-signed}, Theorems \ref{pro:description}, \ref{thm:monotonicity}, \ref{thm:geo-Sobolev-conve} and \ref{thm:monotonicity-manifold} presented in Section \ref{sec:geometry-combi}, has inspired us to understand the graph-to-manifold convergence from the following perspectives.  
\begin{enumerate}
\item See graph-to-manifold approximation as a discrete-to-continuum convergence 
followed by a localization operation: 

In Sections \ref{sec:graph-graphon} and \ref{sec:graphon-manifold}, we show that there exist hidden graph limits generated by weighted geometric graphs on an underlying manifold. 
This is a new perspective for research in graph-to-manifold approximation: when using uniformly sampling graphs to approximate manifolds, we can simply divide the process into two subprocesses: from graphs to graphons (see Section \ref{sec:graph-graphon}), and from graphons to Riemannian manifolds (see Section \ref{sec:graphon-manifold}). 

We can also regard $G_{n,r}\to W_r$ as a discrete-to-continuum convergence, whilst $W_r\to M$ can be seen as a localization process from an analytical perspective. 
In addition, the graph family $\{G_{n,r}\}_{n\ge 1,r>0}$ is, to some extent, analogous to multiparameter persistence, where graphs $\{G_{n,r}\}$ are filtered along two parameters, $n$ and $r$. We can keep track of the graph structure as the scale parameter $r$ varies. 
If we consider the limit of $\{G_{n,r}\}$ where size $n\to+\infty$ and bandwidth $r\to0^+$ are taken simultaneously, it would be convenient to incorporate the dense/sparse graph limit theory into the graph-to-manifold  limiting process. 

\item View graphon $W_r$ as a ``large-scale'' analogy of Riemannian manifold $M$, and view graph $G_{n,r}$ as a coarse-grained approximation of $M$:

Based on Theorems \ref{thm:pers1-diag} and \ref{th:W_r-M}, $W_r$ can be regarded as an $r$-ball bundle of $M$; this suggests that $W_r$ can be viewed as a ``coarse'' approximation or ``large-scale'' form of $M$. 
At the same time, $G_{n,r}$ can be thought of as a coarse-grained approximation of $M$. 

Moreover, as $r$ and $n$ vary, $G_{n,r}$ exhibits certain similarities to the persistent homology theory on graphs or complexes. For example, if for every fixed $r$ there exists some $n_r$ such that $G_{n_r,r}$ is extremely similar to the 1-skeleton of a triangulation of $M$, then for $n \gg n_r$, $G_{n,r}$ is similar to the graphon $W_r$ which serves as a coarse approximation of $M$; whereas when $n \ll n_r$, $G_{n,r}$ resembles a sparse point cloud. 

\item Convergence of combinatorial parameters and geometric quantities:

We prove that the limiting processes $G_{n,r}\to W_r$ and $W_r\to M$ preserve some geometric and analytical quantities, including the spectra, Cheeger constants, $p$-capacities, sphere packing radii, and geometric Sobolev constants. 
However, purely combinatorial parameters, such as the maxcut value, the signed conductance, and some graph-theoretic Sobolev constants, are preserved only in the limiting processes $G_{n,r}\to W_r$; in fact, these combinatorial parameters will ``blow up'' in the limiting processes $W_r\to M$. 
This also explains why certain parameters in some graphs have geometric counterparts—such as the conductance (aka `Cheeger constant')—while others, such as maxcut and the independence number, do not. See Section \ref{sec:geometry-combi}.
\item Relative-energy monotonicity inequality: 

We prove that many combinatorial and geometric parameters on graphons (and manifolds) satisfy a monotonicity formula, which plays an important role in estimating upper and lower bounds for these parameters; see Theorems \ref{thm:monotonicity} and \ref{thm:monotonicity-manifold}. 

\end{enumerate}

\subsection{From graphs to graphons}
\label{sec:graph-graphon}
In this subsection, we clarify the meaning of the limit 
\begin{equation}\label{eq:lim:graph-graphon}
\lim\limits_{n\to\infty}G_{n,r}=W_r
\end{equation}
and then introduce the graphon $p$-Rayleigh quotient with $1\le p< \infty$, where the (generalized)\footnote{In this paper, the generalized graphon will be simply called graphon.} graphon  $W_r:M\times M\to [0,+\infty)$ is defined by   $W_r(x,y)=K_r(\mathrm{dist}(x,y))$ with $\mathrm{dist}(x,y)$ being the distance between $x$ and $y$ in $M$. To rigorously compare discrete graph $G_{n,r}$ and continuum graphon $W_r$, we utilize the $TL^p$-framework which is introduced by Garc\'ia Trillos and Slep\v{c}ev \cite{Garcia-Slepcev,SlepcevThorpe19}, and we shall briefly elaborate on this useful theory in Appendix \ref{Appsec:TLp}. 

With the help of Gamma-convergence and $TL^p$-framework, we can establish the convergence of Sobolev-type constants (see Section \ref{sec:geometry-combi}); to this end, we first present the  Gamma convergence of $p$-Rayleigh quotients in Sections \ref{sec:graph-graphon} and \ref{sec:graphon-manifold} as a preliminary step. 

\begin{defn}
Given $p\ge 1$, the $p$-Rayleigh quotient of 
a weighted graph $G=(V_n,w)$ is a functional $\mathcal{R}_{p,G}:L^p(V_n)\setminus\{0\}\to\R$ defined as 
$$\mathcal{R}_{p,G}(f)=\frac{\sum_{i,j=1}^nw_{ij}|f(x_i)-f(x_j)|^p}{\sum_i\sum_jw_{ij}|f(x_i)|^p},$$
and the $p$-Rayleigh quotient $\mathcal{R}_{p,W}$ of 
a graphon $W:M\times M\to [0,+\infty)$ is a functional $\mathcal{R}_{p,W}:L^p(M)\setminus\{0\}\to\R$ defined as 
$$\mathcal{R}_{p,W}(f):= \frac{\int_{M}\int_{M}|f(x)-f(y)|^pW(x,y)dydx}{\int_{M}|f(x)|^p\int_{M}W(x,y)dydx}.$$
\end{defn}
We say a graph sequence converges to a graphon $W$ in the $TL^p$ sense if the associated edge weight functions $TL^p$-converge to $W$.

\begin{theorem}\label{thm:manifold-sampling-1}
Under Assumptions \ref{ass:1} and \ref{ass:2} where we further assume that the kernel function $K$ appearing in Assumption \ref{ass:1} is continuous, for fixed  threshold 
$r>0$, consider the graph sequence $\{G_{n,r}\}_{n=1}^\infty$ defined by uniform sampling on $M$. Then, we have: 
\begin{enumerate}
\item $\{G_{n,r}\}_{n=1}^\infty$ converges to the  graphon $W_r$  in the $TL^p$ sense, as $n\to\infty$;
\item $\{G_{n,r}\}_{n=1}^\infty$ converges to the  graphon $W_r$  in cut distance, as $n\to\infty$;
\item the $p$-Rayleigh quotient of $G_{n,r}$ converges to that of the graphon $W_r$ as $n$ goes to $\infty$.
\end{enumerate}

\end{theorem}


\subsection{From graphons to Riemannian manifolds}\label{sec:graphon-manifold}
In this subsection, we will explain and clarify the meaning of the limit 
\begin{equation}\label{eq:W_rtoM}
\lim\limits_{r\to0}W_r=M    
\end{equation}
from several perspectives. 
In order to fully illustrate the limit \eqref{eq:W_rtoM}, we need to consider convergence in the sense of both measure and metric, since the  key aspects of Riemannian manifolds are volume measure and Riemannian metric. 
And Theorems \ref{thm:pers1-diag}, \ref{th:W_r-M}  and \ref{thm:metric-converge} show that graphons can recover the Riemannian metrics and volume measures on  manifolds.  

\begin{enumerate}
\item[Perspective 1.] 
Let $\mu_{W_r}$ be a measure on $M\times M$ defined via  $\mu_{W_r}(A\times B):=\int_{A\times B}W_r(x,y)dxdy$, where $A$ and $B$ are measurable subsets of $M$. We would call $\mu_{W_r}$ the \emph{graphon-induced measure} associated with $W_r$. 
\end{enumerate}
\begin{theorem}
\label{thm:pers1-diag}
For any Borel measurable subset $A\subset M$, we have 
$$\lim\limits_{r\to0}\mu_{W_r}(A\times A)=C_{0,k}\mathrm{vol}(A),$$    where  $k
$ is the dimension of $M$, $\vol(A)$ denotes the volume of $A$, and the constant $C_{0,k}$  depending only on $k$ is defined in \eqref{eq:C_0k,C_pk}. 
\end{theorem}

\begin{enumerate}
\item[] 
Since $C_{0,k}$ can be naturally and technically normalized as $1$, it means that, up to a constant, there holds $$\lim\limits_{r\to0}\mu_{W_r}|_{\triangle(M)}\circ i=\mathrm{vol}$$ 
 where $i:\mathcal{B}(M)\to \mathcal{B}(M)\times\mathcal{B}(M)
 $ is defined as $i(A)=A\times A$, $\mathcal{B}$ denotes the Borel sigma-algebra of $M$, $\triangle(M)$ represents the diagonal of $\mathcal{B}(M)\times\mathcal{B}(M)$, and $\vol$ stands for the volume measure on $M$. 
 
\end{enumerate}

Next, we 
present a new interpretation that the limiting process $W_r\to M$ can be regarded as the convergence of $r$-ball bundles to the manifold $M$, as $r$ tends to 0. 

\begin{itemize}
\item[Perspective 2.]  
We define the $r$-ball bundle as the tubular neighborhood of $M$  inside the tangent bundle $TM$, of radius $r$. 
Let $\widetilde{W}_r:TM\to [0,+\infty)$ be defined via   
\begin{equation}\label{eq:tildeWr}
\widetilde{W}_r(x,v)=\begin{cases}W_r(x,\mathrm{exp}_x(v)),&\text{if }|v|\le r\\0,&\text{otherwise}\end{cases}    
\end{equation}
where $\exp_x:T_xM\to M$ is the exponential map. 
When $r$ is sufficiently small, $W_r$ is asymptotic  equivalent to 
$\widetilde{W}_r$ which is defined on the $r$-ball bundle of $M$ in $TM$. Then $\widetilde{W}_r$ induces a measure $\mu_r$ on $TM$, defined via $\mu_r(S):=\int_S \widetilde{W}_r$,  which is concentrated on the $r$-ball bundle. We call $\mu_r$ the bundle-derived measure induced by $W_r$.
\end{itemize}

\begin{theorem}\label{th:W_r-M}
The bundle-derived measure $\mu_r$ induced by $W_r$ 
weak-star converges to a measure $\tilde{\mu}$ with $\mathrm{supp}(\tilde{\mu})=M$ and $\tilde{\mu}|_M$ coincides with the volume measure on $M$ up to a constant. 

Furthermore, for any $f\in \mathbb{W}^{1,p}(M)$, we have the convergence for Rayleigh quotient:
$$\lim\limits_{r\to0^+}\frac{1}{r^p}\frac{\int_{M}\int_{M}|f(x)-f(y)|^pW_r(x,y)dydx}{\int_{M}|f(x)|^p\int_{M}W_r(x,y)dydx}=\frac{C_{p,k}\int_{M}|\nabla f(x)|^pdx}{C_{0,k}\int_{M}|f(x)|^pdx}$$
where $C_{p,k}$ and $C_{0,k}$ are defined in \eqref{eq:C_0k,C_pk}, $\mathbb{W}^{1,p}(M)$ denotes the Sobolev space of $L^p$ functions on $M$ whose weak gradient is also in $L^p$. 
\end{theorem}


\begin{remark}
Note that in Theorem \ref{th:W_r-M}, we additionally establish a convergence on the $p$-Rayleigh quotients. This can be used to further prove the convergences of graphon parameters. 
We may refer to the class of graph parameters that have geometric analogs as ‘graph-geometric’ parameters, and the class that do not have geometric analogs as ‘graph-combinatorial’ parameters. 
Both types of parameters are preserved under the specific graph-to-graphon convergence $G_{n,r}\to W_r$ presented in Theorem \ref{thm:manifold-sampling-1}, while only graph-geometric parameters survive under the graphon-to-manifold convergence $W_r\to M$; in fact, graph-combinatorial parameters undergo blowup under the $W_r$-to-$M$ convergence; see Section \ref{sec:geometry-combi}. 
\end{remark}

\begin{itemize}
\item[Perspective 3.] Another viewpoint on the graphon-to-manifold process is that the following metric convergence. 
Before going into the theorem, we first introduce the following graphon-distance defined as
$$\mathrm{dist}_{W}(x,y):=\sup_{Z\subset J\setminus \{x,y\}:\mu(Z)=0}\inf\limits_{\substack{x_0,\cdots,x_{n}\in J\setminus Z\\ n\ge 1, x_0=x,x_n=y}}\sum_{i=0}^{n-1}\frac{1}{W(x_i,x_{i+1})}
$$
which denotes the ``essential" shortest path distance between $x$ and $y$ in $J$. 
The concept of this distance will also be used to characterize the sphere packing radius on graphon, see e.g. Section \ref{sec:sphere-radius}.
\end{itemize}
\begin{theorem}\label{thm:metric-converge}
Let $\mathrm{dist}_{M}(x,y)$ denote the shortest path distance between $x$ and $y$ in the Riemannian manifold $M$. Then,
$$\lim\limits_{r\to0^+}\frac{\mathrm{dist}_{W_r}(x,y)}{r^{k-1}}=\tilde{C}\,\mathrm{dist}_{M}(x,y)$$
where $\tilde{C}:=\big(\sup_{t>0}tK(t)\big)^{-1}$,   $K$ is defined in Assumption \ref{ass:1}, and $k$ is the dimension of $M$.
\end{theorem}

\begin{remark}
Theorems \ref{thm:pers1-diag} and \ref{th:W_r-M}, arisen from Perspectives 1 and 2, demonstrate‌ that the measures induced by $W_r$ in certain ways converge to the volume measure on $M$. 
While, Theorem \ref{thm:metric-converge}, arisen from  Perspective 3, shows that metrics induced by $W_r$ in a certain sense recover the Riemannian metric of $M$. 
\end{remark}

If we write for any $f\in \mathbb{W}^{1,p}(M)$, 
$$\mathcal{R}_{p,W_r}(f):=\frac{\int_{M}\int_{M}|f(x)-f(y)|^pW_r(x,y)dydx}{\int_{M}|f(x)|^p\int_{M}W_r(x,y)dydx} \text{ and }\mathcal{R}_{p,M}(f):=\frac{\int_{M}|\nabla f(x)|^pdx}{\int_{M}|f(x)|^pdx}$$
then Theorem \ref{th:W_r-M} implies that $r^{-p}\mathcal{R}_{p,W_r}$ converges to $\frac{C_{p,k}}{C_{0,k}}\mathcal{R}_{p,M}$ pointwisely when $r$ tends to $0$. In fact, one can derive the Gamma-convergence property:
\begin{prop}\label{prop:gamma-Rp}
For any $p\in[1,+\infty)$, the functional $r^{-p}\mathcal{R}_{p,W_r}$ Gamma-converges to $\frac{C_{p,k}}{C_{0,k}}\mathcal{R}_{p,M}$ as $r$ tends to $0^+$.
\end{prop}

Let 
\begin{equation}\label{eq:cRpWr}
c(\mathcal{R}_{p,W_r})=\inf_{f\ne \text{ const}}\sup_{c\in\R}\mathcal{R}_{p,W_r}(f-c) \text{ and }c(\mathcal{R}_{p,M})=\inf_{f\ne \text{ const}}\sup_{c\in\R}\mathcal{R}_{p,M}(f-c).    
\end{equation}
It is not difficult to show that $c(\mathcal{R}_{2,W_r})$ is the second smallest eigenvalue of the Laplacian on $W_r$, while $c(\mathcal{R}_{1,W_r})$ equals the conductance 
on $W_r$; see Section \ref{sec:geometry-combi}. 
In the manifold setting, for $p\in(1,\infty)$, $c(\mathcal{R}_{p,M})$ is the second smallest eigenvalue of the $p$-Laplacian on $M$, while $c(\mathcal{R}_{1,M})$ equals the Cheeger constant on $M$. 

\begin{theorem}
\label{th:gamma-function}
For any $p,q\ge 1$, 
we have 
\begin{equation}\label{eq:gamma-Cpq-Wr-M}
\lim_{r\to0^+}\frac{c(\mathcal{R}_{p,W_r})^{\frac1p}}{c(\mathcal{R}_{q,W_r})^{\frac1q}}=\Big(\frac{C_{p,k}}{C_{0,k}}\Big)^{\frac1p}\Big(\frac{C_{q,k}}{C_{0,k}}\Big)^{-\frac1q}\frac{c(\mathcal{R}_{p,M})^{\frac1p}}{c(\mathcal{R}_{q,M})^{\frac1q}}
\end{equation}
where $C_{0,k}$ and $C_{p,k}$ are defined in \eqref{eq:C_0k,C_pk}. Furthermore, the factor $\Big(\frac{C_{p,k}}{C_{0,k}}\Big)^{\frac1p}\Big(\frac{C_{q,k}}{C_{0,k}}\Big)^{-\frac1q}$ converges to 
\begin{equation}\label{eq:Apq}\big(\frac{1}{\sqrt{\pi}}\Gamma(\frac{p+1}{2})\big)^{\frac1p}\big(\frac{1}{\sqrt{\pi}}\Gamma(\frac{q+1}{2})\big)^{-\frac1q}\end{equation} 
as the dimension $k\to\infty$, where $\Gamma(\cdot)$ denotes the standard Gamma-function.
\end{theorem}


The results in Sections \ref{sec:graph-graphon} and \ref{sec:graphon-manifold} then derive:
$$\xymatrix{G_{n,r}\ar[rrr]^{\text{\small as }{\displaystyle n\to+\infty}}_{\substack{\text{\small converges to}\\ \text{(by Theorem \ref{thm:manifold-sampling-1})}}}& & &W_r\ar[rrr]^{\text{\small as }{\displaystyle r\to0^+}}_{\substack{\text{\small converges to}\\ \text{(by Theorems \ref{thm:pers1-diag},  \ref{th:W_r-M} and \ref{thm:metric-converge})}}}& & & M}$$
In particular, our result \eqref{eq:gamma-Cpq-Wr-M} shows the dimension dependence. 

Our 
insight in this paper is that we actually establish relations among graphs $G_{n,r}$, graphons $W_r$ and manifolds $M$, and our results reinforce the central role of graphons. 
Together with results in Section \ref{sec:geometry-combi}, we will see that continuous versions of some graph parameters cannot be realized via Riemannian manifolds, but can be realized via graphons.

\subsection{Detailed explanations and Proofs of theorems in Section \ref{sec:hidden-limit}
}
\label{sec:explain-proof}
We use notions and tools from the calculus of variations and optimal transport, such as Gamma convergence and $TL^p$ convergence. 
Gamma convergence was introduced by De Giorgi in 1970’s to study limits of variational problems. We refer to \cite{Dal93,Braides02}  for an in-depth introduction to Gamma-convergence. %

Given a compact Riemannian manifold $M$, denote by $M^{\mathbb{N}}$ the set of sequences $\{(x_i)_{i=1}^{+\infty}:x_i\in M\}$. 
By Kolmogorov's existence theorem, there exists a unique probability measure $\mathbb{P}$ on $M^{\mathbb{N}}$ such that $\mathbb{P}(\prod_{j\in J} B_j\prod_{j\in \mathbb{N}\setminus J}M)=\prod_{j\in J}\mu(B_j)$ for any finite set $J\subset \mathbb{N}$ and measurable set $B_j\subset M$, where $\mu$ is the normalized volume measure on $M$, i.e., $\mu=\vol(M)^{-1}\vol$. 

We say $(x_i)_{i=1}^{+\infty}$ is \emph{admissible}, if there exists a sequence of transport maps $\left\{ T_n \right\}_{n \in \N}$ from $\mu$ to $\mu_n$ (i.e., $T_{n \sharp} \mu = \mu_n$) and such that $\|\mathrm{Id} - T_n\|_\infty\to0$ as $n\to\infty$, where $\mu_{n}:=\frac1n\sum_{i=1}^n\delta_{x_i}$ and $\delta_{x_i}$ denotes the Dirac measure at $x_i$. The first result we should mention from \cite[Theorem 1]{W8L8} is that the admissible sequences are $\mathbb{P}$-a.e. in $M^{\mathbb{N}}$, i.e., 
\begin{equation}
\mathbb{P}\left((x_i)_{i=1}^{+\infty}\in M^{\mathbb{N}}: (x_i)_{i=1}^{+\infty}\text{ is admissible}\right)=1.
\end{equation}
Thus, in the sequel, we always assume that the sequence $(x_i)_{i=1}^{+\infty}$ is admissible and the associated transport maps $(T_n)_{n=1}^{+\infty}$ are also prescribed.

\begin{proof}[Proof of Theorem \ref{thm:manifold-sampling-1}] 
Given $r>0$, we shall prove that the sequence of stepgraphons $W_{G_{n,r}}$ constructed from $G_{n,r}$ converges to $W_r$ in $L^1$, as $n$ goes to infinity. 
It follows from the compactness of $M$ that 
$$c:=\inf_{x\in M}\frac{\mathrm{vol}(B_r(x))}{\mathrm{vol}(M)}=\min_{x\in M}\frac{\mathrm{vol}(B_r(x))}{\mathrm{vol}(M)}>0.$$
Note that, for any vertex $i$, the expectation‌ of the degree of $i$ is at least $cn$.  
Thus, the number of edges equals $\frac12\sum_{i\in V}\deg(i)\ge \frac{c}{2}n^2$ in expectation. This indicates that $\{G_{n,r}\}_{n\ge1}$ is a sequence of dense graphs, $\forall r>0$. 

Consider a labeled partition $\{M_i\}$ of $M$ such that  $M_i\cap V_n=\{ x_i\}$ and $\vol(M_i)=\vol(M)/n$, $i=1,\cdots,n$, and $\max_{1\le i\le n}\mathrm{diam}(M_i)=o(1)$, that is, the diameter of each $M_i$ converges to 0 as $n$ tends to $\infty$.  As a natural graphon representation of $G_{n,r}$, the stepgraphon $W_{G_{n,r}}$ is defined via $W_{G_{n,r}}(x,y)=W_r(x_i,x_j)$ whenever $x\in M_i$ and $y\in M_j$. 
Now we will estimate the $L^1$-difference between $G_{n,r}$ and $W_r$ via the equality
\begin{align*}
\|W_{G_{n,r}}-W_r\|_{1}:&=\int_{M\times M} |W_{G_{n,r}}(x,y)-W_r(x,y)|dxdy \\&= \sum_{i,j=1}^n \int_{M_i\times M_j} |W_{G_{n,r}}(x,y)-W_r(x,y)|dxdy    
 \\&= \sum_{i,j=1}^n \int_{M_i\times M_j} |W_r(x_i,x_j)-W_r(x,y)|dxdy \,.   
\end{align*}

Since $W_r$ is continuous 
on the compact space $M\times M$, we have for any $\varepsilon>0$, there exists $\delta>0$ such that $\mathrm{dist}(x,x_i)<\delta$ and $\mathrm{dist}(y,x_i)<\delta$ imply $|W_r(x_i,x_j)-W_r(x,y)|<\varepsilon$.  
When $n$ is sufficiently large, and the partition is sufficiently fine such that $\max_{1\le i\le n}\mathrm{diam}(M_i)<\delta$, we have 
$$\sum_{i,j=1}^n \int_{M_i\times M_j} |W_r(x_i,x_j)-W_r(x,y)|dxdy\le \sum_{i,j=1}^n  \vol(M_i)  \vol(M_j)\varepsilon =\vol(M)^2\varepsilon$$
which implies $\|W_{G_{n,r}}-W_r\|_1\to 0$ as $n\to +\infty$. Thus $$\lim\limits_{n\to\infty}\delta_\square(W_{G_{n,r}},W_r)=0.$$ 
From the above proof, it can be checked that in a similar manner, for $p\in[1,\infty]$, $\|W_{G_{n,r}}-W_r\|_p\to 0$, $r\to0$. 

We refer to \cite{Garcia-Slepcev} that  $\mu_n$ weakly$^*$ converges to $\mu$ and there exists $T_n$ such that $(T_n)_\#(\mu)=\mu_n$ and 
$$\int_{M\times M}|W_r(x,y)-W_{G_{n,r}}(T_nx,T_ny)|^pd\mu(x) d\mu(y)\to0,\;n\to\infty$$
and these conclude the convergence in the sense of $TL^p$.  

Next, we prove the convergence of the $p$-Rayleigh quotient. 
Taking  the limit of the Riemann sum, we have for any Riemannian integrable function $f$, there holds $$
\lim\limits_{
n\to+\infty}\frac{\sum\limits_{i,j=1}^n|f(x_i)-f(x_j)|^pW_{G_{n,r}}(x_i,x_j)}{n^2}=\int_{M}\int_{M}|f(x)-f(y)|^pW_r(x,y)d\mu(y)d\mu(x)$$
where $d\mu(x):=dx/\vol(M)$, and 
$$
\lim\limits_{
n\to+\infty}\frac{1}{n}\sum_{i=1}^n|f(x_i)|^p\frac1n\sum_{j=1}^nW_r(x_i,x_j)=\int_{M}|f(x)|^p\int_{M}W_r(x,y)d\mu(y)d\mu(x).$$
Then by setting $w_{ij}=W_r(x_i,x_j)$ and $\mu_i=\sum_{j=1}^nW_r(x_i,x_j)$, we have
\begin{align*}
\lim\limits_{
n\to+\infty}\frac{\sum_{i,j=1}^n|f(x_i)-f(x_j)|^pw_{ij}}{\sum_i|f(x_i)|^p\mu_i}
&=\lim_{n\to+\infty} \frac{\frac{1}{n^2}\sum_{i,j=1}^n|f(x_i)-f(x_j)|^pW_r(x_i,x_j)}{\frac{1}{n^2}\sum_{i=1}^n|f(x_i)|^p\sum_{j=1}^nW_r(x_i,x_j)}
\\&=\frac{\int_{M}\int_{M}|f(x)-f(y)|^pW_r(x,y)d\mu(y)d\mu(x)}{\int_{M}|f(x)|^p\int_{M}W_r(x,y)d\mu(y)d\mu(x)}
\\&=\frac{\int_{M}\int_{M}|f(x)-f(y)|^pW_r(x,y)dydx}{\int_{M}|f(x)|^p\int_{M}W_r(x,y)dydx}
\end{align*}
i.e., $\mathcal{R}_{p,G_{n,r}}(f|_{V_n})\to \mathcal{R}_{p,W_r}(f)$ as $n\to\infty$.  
The proof of the pointwise convergence is complete. 
\end{proof}

Notice that the limit \eqref{eq:lim:graph-graphon} 
in our context means that the  stepgraphon $W_{G_{n,r}}$  corresponding to $G_{n,r}$ converges to $W_r$ in cut distance (and also in $TL^p$ sense) 
as $n\to+\infty$. 
Using the same approaches proposed by Garc\'ia Trillos and Slep\v{c}ev \cite{Garcia-Slepcev,SlepcevThorpe19},  it is standard to show that $\mathcal{R}_{p,G_{n,r}}$ Gamma-converges to $\mathcal{R}_{p,W_r}$ as $n\to\infty$.

\begin{proof}[Proof of Theorem \ref{thm:pers1-diag}]
We shall prove that for any Riemannian measurable set $A\subset M$, 
\begin{equation}\label{eq:AtimesAtovol(A)}
\lim\limits_{r\to0^+}\int_{A\times A}W_r(x,y)dxdy=C_{0,k}\mathrm{vol}(A)    
\end{equation}
where $k=\dim M$. 

\medskip
\noindent{\bf Step 1}. $A$ is a nonempty open subset of $M$.

Fixed $x\in A$ and $\varepsilon>0$, there exists $\delta>0$ such that $B(x,\delta)\subset A$ and the exponential map $\exp_x $ induces a diffeomorphism from $\{v\in T_xM:|v|<\delta\}$ to $ B(x,\delta)$ satisfying 
$$1-\varepsilon<|\det(d \exp_x)_v|<1+\varepsilon\;\text{ and }\;|v|=\mathrm{dist}(x,\exp_x(v))$$ whenever $|v|<\delta$. Without loss of generality, we further assume $\delta<\mathrm{inj}(M)$, where $\mathrm{inj}(M)$ denotes the  injectivity radius of $M$. 
Consider $$\int_{A}W_r(x,y)dy=\int_{A}K_r(\mathrm{dist}(x,y))dy=\int_{B(x,\delta)}K_r(\mathrm{dist}(x,y))dy+\int_{A\setminus B(x,\delta)}K_r(\mathrm{dist}(x,y))dy.$$
Note that 
\begin{align*}
\int_{B(x,\delta)}K_r(\mathrm{dist}(x,y))dy&=\int_{|v|<\delta}K_r(\mathrm{dist}(x,\exp_x(v)))\,|\det(d \exp_x)_v|dv
\\&\le (1+\varepsilon)\int_{|v|<\delta}K_r(|v|)dv=(1+\varepsilon)\int_{|u|<\frac\delta r}K_r(r|u|)r^kdu
\\&=(1+\varepsilon)\int_{|u|<\frac\delta r}K(|u|)du \le (1+\varepsilon)\int_{\R^k}K(|u|)du=(1+\varepsilon)C_{0,k}
\end{align*}
and 
\begin{align}
\int_{A\setminus B(x,\delta)}K_r(\mathrm{dist}(x,y))dy\le \int_{M\setminus B(x,\delta)}K_r(\mathrm{dist}(x,y))dy
\le \vol(M)\,K_r(\delta)
\to 0,\; \text{ as }r\to0^+
\label{eq:C'int}
\end{align}
where $\lim_{r\to0}K_r(\delta)=0$ can be easily derived from the condition on $K_r$; see Assumption \ref{ass:1}. 
This implies $\limsup_{r\to0^+}\int_{A}W_r(x,y)dy\le (1+\varepsilon)C_{0,k}$. Similarly,   
\begin{align*}
\int_{B(x,\delta)}K_r(\mathrm{dist}(x,y))dy&=\int_{|v|<\delta}K_r(\mathrm{dist}(x,\exp_x(v)))\,|\det(d \exp_x)_v|dv
\\&\ge (1-\varepsilon)\int_{|v|<\delta}K_r(|v|)dv=(1-\varepsilon)\int_{|u|<\frac\delta r}K_r(r|u|)r^kdu
\\&=(1-\varepsilon)\int_{|u|<\frac\delta r}K(|u|)du \to (1-\varepsilon)\int_{\R^k}K(|u|)du=(1-\varepsilon)C_{0,k}, \text{ as }r\to 0^+
\end{align*}
and together with \eqref{eq:C'int} implies $\liminf_{r\to0^+}\int_{A}W_r(x,y)dy\ge (1-\varepsilon)C_{0,k}$. 
By the arbitrariness of $\varepsilon$, we have 
\begin{equation}\label{eq:int_AW_r=C0k}
\lim_{r\to0^+}\int_{A}W_r(x,y)dy=C_{0,k}.    
\end{equation}

Let $C'=\sup_{x\in M,|v|<\mathrm{inj}(M)}|\det(d \exp_x)_v|>0$. We have 
\begin{align*}
\int_{A}W_r(x,y)dy&\le \int_{B(x,\mathrm{inj}(M))}K_r(\mathrm{dist}(x,y))dy+\int_{A\setminus B(x,\mathrm{inj}(M))}K_r(\mathrm{dist}(x,y))dy
\\&\le \int_{|v|<\mathrm{inj}(M)}K_r(\mathrm{dist}(x,\exp_x(v)))\,|\det(d \exp_x)_v|dv+\vol(A)\,K_r(\mathrm{inj}(M))
\\&\le C'\int_{|v|<\mathrm{inj}(M)} K_r(|v|)dv+\vol(M)\,K_r(\mathrm{inj}(M))
\\&\le C'\int_{\R^k} K_r(|v|)dv+\vol(M)\,\frac{1}{r^k}K\big(\frac{\mathrm{inj}(M)}{r}\big)
\\&\le C'\int_{\R^k} K(|u|)du+\vol(M)\sup_{r>0}\frac{1}{r^k}K\big(\frac{\mathrm{inj}(M)}{r}\big)
\\&=  C'C_{0,k}+ \frac{\vol(M)}{\big(\mathrm{inj}(M)\big)^k} \widehat{C}
\end{align*}
for any $r>0$ and any $x\in M$, where $\widehat{C}:=\sup_{t>0} t^kK(t)$ is a positive number (follows from Assumption \ref{ass:1}).  Then, by \eqref{eq:int_AW_r=C0k} and Lebesgue's dominated convergence theorem, we have
$$\lim\limits_{r\to0^+}\int_{A}\int_{A}W_r(x,y)dydx=\int_{A}C_{0,k}dx=C_{0,k}\mathrm{vol}(A)$$
which verifies \eqref{eq:AtimesAtovol(A)} for any open set $A$. 

\medskip
\noindent{\bf Step 2}. $A$ is a nonempty proper closed subset of $M$.

For any nonempty proper closed subset $A\subset M$, its complement $A^c$ is a nonempty open subset of $M$, and thus $$\lim\limits_{r\to0^+}\int_{A^c\times A^c}W_r(x,y)dxdy=C_{0,k}\mathrm{vol}(A^c).$$
For any $x\in A^c$, it follows from the closedness of $A$ that $\delta_x:=\inf_{y\in A}\mathrm{dist}(x,y)>0$. Then, similar to \eqref{eq:C'int}, it can be verified that 
$$\int_AW_r(x,y)dy\le \int_{M\setminus B(x,\delta_x)}W_r(x,y)dy=\int_{M\setminus B(x,\delta_x)}K_r(\mathrm{dist}(x,y))dy\to 0,\;\text{ as }r\to0$$
and analogously, we have
$$\lim\limits_{r\to0^+}\int_{A\times A^c}W_r(x,y)dxdy=\lim\limits_{r\to0^+}\int_{A^c\times A}W_r(x,y)dxdy=\lim\limits_{r\to0^+}\int_{A^c}\Big(\int_{A}W_r(x,y)dy\Big)dx=0.$$
Since $M$ is itself an open subset of $M$, we have 
$$\lim\limits_{r\to0^+}\int_{M\times M}W_r(x,y)dxdy=C_{0,k}\mathrm{vol}(M).$$
Together with all the equalities above, we obtain that \eqref{eq:AtimesAtovol(A)} also holds for any closed subset $A$. 

\medskip
\noindent{\bf Step 3}. $A$ is a nonempty Borel  subset of $M$.

Now, for any nonempty Borel measurable subset $A\subset M$, for any $\varepsilon>0$, there exists a closed subset $A_0\subset A$ and an open neighborhood $A_1\supset A$ such that $\vol(A_1)<(1+\varepsilon)\vol(A)$ and $\vol(A_0)>(1-\varepsilon)\vol(A)$. 
Therefore, 
$$\limsup\limits_{r\to0^+} \lim\limits_{r\to0^+}\int_{A\times A}W_r(x,y)dxdy\le \lim\limits_{r\to0^+}\int_{A_1\times A_1}W_r(x,y)dxdy=C_{0,k}\mathrm{vol}(A_1)<(1+\varepsilon)C_{0,k}\vol(A)$$
and 
$$\liminf\limits_{r\to0^+} \lim\limits_{r\to0^+}\int_{A\times A}W_r(x,y)dxdy\ge \lim\limits_{r\to0^+}\int_{A_0\times A_0}W_r(x,y)dxdy=C_{0,k}\mathrm{vol}(A_0)>(1-\varepsilon)C_{0,k}\vol(A).$$
Taking $\varepsilon\to0^+$, we derive that \eqref{eq:AtimesAtovol(A)} also holds for any Borel measurable subset $A$. 

Finally, it follows from Step 3 that \eqref{eq:AtimesAtovol(A)} holds for any Riemannian measurable subset $A$, as there exists a Borel set $A'\supset A$ with $\vol(A'\setminus A)=0$. 
\end{proof}

\begin{proof}[Proof of Theorem \ref{th:W_r-M}] Let $\mu_M$ denote the Riemannian (volume) measure on $M$. 
It suffices to show that there exists constant $C>0$ such that 
$$\lim_{r\to0^+}\int_{TM} h(x,v)d\mu_r(x,v)=
C\int_{M} h(x,0)d\mu_M(x)$$
for any smooth test function $h:TM\to\R$ with compact support. 
Since $M$ is compact and $h$ is continuous, the function $v\mapsto \max\limits_{x\in M}\big|h(x,v)-h(x,0)\big|$ must be continuous. Then, for any $\varepsilon>0$, there exists $\delta>0$ such that 
\begin{equation}\label{eq:Lebesgue}
\int_M\big|h(x,v)-h(x,0)\big|d\mu_M(x)\le \mu(M) \max\limits_{x\in M}\big|h(x,v)-h(x,0)\big|<\varepsilon    
\end{equation}
whenever $|v|\le \delta$. 
For sufficiently small $r>0$, we have 
\begin{equation}\label{eq:cut-off-estimate}
\Big|\int_{|v|\ge \delta}K_r(|v|)dv\Big|=\int_{|u|\ge\frac\delta r}K_r(r|u|)r^kdu=\int_{|u|\ge\frac\delta r}K(|u|)du<\varepsilon   
\end{equation} 

Then taking $\widehat{W}_r(x,v) :=K_r(|v|)$, we obtain
{
\begin{align*}
&\left|\int_{TM} h(x,v)\widehat{W}_r(x,v) d\mu(x)dv -C_{0,k}\int_M  h(x,0)d\mu(x)\right|
\\=~&\left|\int_{M\times \R^k} h(x,v)K_r(|v|) d\mu(x)dv -\int_{\R^k}K_r(|v|)dv
\int_M  h(x,0)d\mu(x)\right|
\\=~&\left|\int_{\R^k}\int_{M} (h(x,v)-h(x,0))d\mu(x)K_r(|v|) dv \right|
\\=~&
\left|\int_{|v|> \delta}\int_{M} (h(x,v)-h(x,0))d\mu(x)K_r(|v|) dv+ \int_{|v|\le \delta}\int_{M} (h(x,v)-h(x,0))d\mu(x)K_r(|v|) dv\right|
\\ \le ~&\int_{|v|> \delta}\int_{M} \big|h(x,v)-h(x,0)\big|d\mu(x)K_r(|v|) dv + \int_{|v|\le \delta}\int_{M} \big|h(x,v)-h(x,0)\big|d\mu(x)K_r(|v|) dv 
\\ \le ~&\int_{|v|> \delta}\int_{M} 2Cd\mu(x)K_r(|v|) dv + \int_{|v|\le \delta}\varepsilon K_r(|v|) dv 
\\ = ~&2C\mu(M)\int_{|v|> \delta}  K_r(|v|) dv + \varepsilon\int_{|v|\le \delta} K_r(|v|) dv 
\\ \le ~&2C\mu(M)\varepsilon + \varepsilon C_{0,k} 
\end{align*}
}
where $C:=\sup\limits_{x,v}|h(x,v)|<+\infty$, and we used both \eqref{eq:Lebesgue} and \eqref{eq:cut-off-estimate}.

Therefore, we obtain
\[\lim_{r\to0^+}\int_{TM} h(x,v)\widehat{W}_r(x,v) d\mu(x)dv =C_{0,k}\int_M  h(x,0)d\mu(x).\]

Recalling \eqref{eq:tildeWr} for the definition of $\widetilde{W}_r$, it is clear that for any $r\in(0,\mathrm{inj}(M))$ and for any $v$ with  $|v|\le r$, $$\widetilde{W}_r(x,v)=W_r(x,\mathrm{exp}_x(v))=K_r(\mathrm{dist}(x,\mathrm{exp}_x(v)))=K_r(|v|)=\widehat{W}_r(x,v).$$ Consequently, when $0<r<\mathrm{inj}(M)$, $$\widehat{W}_r(x,v)-\widetilde{W}_r(x,v)=\begin{cases}0,&\text{if }|v|\le r\\ \widehat{W}_r(x,v),&\text{otherwise}.\end{cases}$$

By using $$\overline{K}_r(|v|)=\begin{cases}K_r(|v|),&\text{if }  |v|\le r\\ 0,&\text{otherwise}\end{cases}$$ instead of $K_r(|v|)$, and using $C:=\int_{\R^k}\overline{K}(|v|)dv=\int_{|v|\le 1}K(|v|)dv$ instead of $C_{0,k}$, we similarly derive 
\begin{equation}\label{eq:Wtilde}
\lim_{r\to0^+}\int_{TM} h(x,v)\widetilde{W}_r(x,v) d\mu(x)dv =C\int_M  h(x,0)d\mu(x).    
\end{equation}
This concludes the proof of the weak-star convergence of $\mu_r$ to $\mu_M$ up to the constant $C$, where $d\mu_r(x,v)=\widetilde{W}_r(x,v) d\mu(x)dv$. 

Now we focus on the convergence about the $p$-Rayleigh quotient $\mathcal{R}_{p,W_r}$. 
Choose \(\varepsilon>0\) smaller than the injectivity radius of \(M\). 
For \(y\in B_\varepsilon(x)\), write
\[
y=\exp_x (v),\;\text{ with } v\in T_xM\text{ and } |v|<\varepsilon.
\]
In geodesic normal coordinates, $\mathrm{dist}(x,y)=|v|$, and the Riemannian volume element satisfies
\[
dy=|\det(d \exp_x)_v|\,dv=(1+O(|v|^2))dv,
\]

For any $C^2$-smooth $f$, we have 
\[
f(y)-f(x)=\langle \nabla f(x),v\rangle+O(|v|^2).
\]
The proof of 
\[\lim_{r\to0^+}\frac{1}{r^p}\int_M\int_{M\setminus B(x,\varepsilon)} |f(x)-f(y)|^pK_r(\mathrm{dist}(x,y))\,dy\,dx=0 \]
is standard and thus we omit it. Moreover, 
\begin{align*}
&\frac{1}{r^p}\int_M\int_{B(x,\varepsilon)} |f(x)-f(y)|^pK_r(\mathrm{dist}(x,y))\,dy\,dx\\
=\,&\frac{1}{r^p}\int_M\int_{|v|<\varepsilon}\left|\langle \nabla f(x),v\rangle+O(|v|^2)\right|^pK_r(|v|)\bigl(1+O(|v|^2)\bigr)\,dv\,dx
\\=\,&\frac{1}{r^p}\int_M\int_{|u|<\frac{\varepsilon}{r}}\left|r\langle \nabla f(x),u\rangle+O(r^2|u|^2)\right|^pK_r(r|u|)\bigl(1+O(r^2|u|^2)\bigr)\,r^kdu\,dx
\\=\,&\int_M\int_{|u|<\frac{\varepsilon}{r}}\left|\langle \nabla f(x),u\rangle+O(r|u|^2)\right|^pK(|u|)\bigl(1+O(r^2|u|^2)\bigr)\,du\,dx
\\=\,&\int_M\int_{|u|<\varepsilon/r}
|\langle \nabla f(x),u\rangle|^pK(|u|)\,du\,dx+o(1)
\\\to \,&\int_M\int_{\R^k}
|\langle \nabla f(x),u\rangle|^pK(|u|)\,du\,dx\;\text{ as }r\to0,
\end{align*}
where we substituted \(v=ru\). 
Choose an orthonormal basis of \(T_xM\) such that
\[
\frac{\nabla f(x)}{|\nabla f(x)|}=\mathbf{e}_1.
\]
Then
\[
\langle \nabla f(x),u\rangle=|\nabla f(x)|\,u_1.
\]
Therefore,
\begin{align*}
\lim_{r\to0^+}\frac{1}{r^p}\int_M\int_M |f(x)-f(y)|^pK_r(\mathrm{dist}(x,y))\,dy\,dx&=\int_M\int_{\R^k}
|\nabla f(x)|^p |u_1|^pK(|u|)\,du\,dx
\\&=C_{p,k}\int_M|\nabla f(x)|^p\,dx.
\end{align*}
Similarly, we have
\begin{align*}
\lim_{r\to0^+}\int_M\int_M |f(x)|^p K_r(\mathrm{dist}(x,y))\,dy\,dx
&= \lim_{r\to0^+}\int_M\int_{B(x,\varepsilon)} |f(x)|^p K_r(\mathrm{dist}(x,y))\,dy\,dx\\
&= \lim_{r\to0^+}\int_M\int_{v\in T_xM,\,|v|<\varepsilon} |f(x)|^p K_r(|v|)\,|\det(d \exp_x)_v|\,dv\,dx\\
&= \lim_{r\to0^+}\int_M |f(x)|^p \left( \int_{v\in T_xM,\,|v|<\varepsilon} K_r(|v|)\,\bigl(1+O(|v|^2)\bigr)\,dv \right) dx\\
&= \lim_{r\to0^+}\int_M |f(x)|^p \left( \int_{u\in\R^k,\,|u|<\varepsilon/r} K(|u|)\,\bigl(1+O(r^2|u|^2)\bigr)\,du \right) dx\\
&= \left(\int_{\R^k} K(|u|)\,du\right) \int_M |f(x)|^p\,dx\\
&= C_{0,k}\int_M |f(x)|^p\,dx.
\end{align*}
These complete the proof of the pointwise convergence $r^{-p}\mathcal{R}_{p,W_r}\to C_{p,k}C_{0,k}^{-1}\mathcal{R}_{p,M}$ as $r$ tends to 0. 
\end{proof}

\begin{proof}[Proof of Proposition \ref{prop:gamma-Rp}]
Note that the proof of Theorem  \ref{th:W_r-M}  also completes the proof of the limsup inequality. 

While, the proof of liminf inequality is very similar to the work of Garc\'ia Trillos and Slep\v{c}ev  \cite{Garcia-Slepcev}, and thus we omit the detail. 
\end{proof}

\begin{proof}[Proof of Theorem \ref{thm:metric-converge}]
Let 
\[S:=\sup_{t>0} tK(t)=\frac{1}{\tilde{C}}.
\]
By Assumption \ref{ass:1} for $K$, one has $0<S<\infty$.  
Since 
\begin{align*}
W_r(x,y)=K_r(\mathrm{dist}_M(x,y))=\frac{1}{r^k}K\big(\frac{\mathrm{dist}_M(x,y)}{r}\big),
\end{align*}
Theorem \ref{thm:metric-converge} can be reduced to the following limit equality (to be proved):
\[ \lim_{r\to0}\sup_{Z\subset M:\mu(Z)=0}\inf\limits_{\substack{x_1,x_2,\cdots,x_{n-1}\in M\setminus Z\\ n\ge 1, x_0:=x,x_n:=y}} \sum_{i=0}^{n-1}\frac{r}{K\big(\frac{\mathrm{dist}_M(x_i,x_{i+1})}{r}\big)}=\tilde{C}\,\mathrm{dist}_{M}(x,y).\]

For any $Z$ with $\mu(Z)=0$, for any chain $\{x_0,\dots,x_n\}\subset M$ with $x_0=x$, $x_n=y$, and $x_i\not\in Z$ $\forall i\in\{1,\cdots,n-1\}$, if $\dist_M(x_i,x_{i+1})>0$, then with $t_i=\dist_M(x_i,x_{i+1})/r$,
\[
\frac{r}{K(\dist_M(x_i,x_{i+1})/r)}
=
\frac{\dist_M(x_i,x_{i+1})}{t_iK(t_i)}
\ge
\frac{\dist_M(x_i,x_{i+1})}{S}
=
\tilde{C}\, \dist_M(x_i,x_{i+1}),
\]
because $t_iK(t_i)\le S$. If $\dist_M(x_i,x_{i+1})=0$, then the above inequality is trivial. Hence
\[
\sum_{i=0}^{n-1}\frac{r}{K(\dist_M(x_i,x_{i+1})/r)}
\ge
\tilde{C}\sum_{i=0}^{n-1}\dist_M(x_i,x_{i+1})
\ge
\tilde{C}\, \dist_M(x,y).
\]
Therefore 
\begin{equation}\label{eq:liminf>=dist}
\liminf_{r\to0}\sup_{Z\subset M:\mu(Z)=0}\ \inf\limits_{\substack{x_1,x_2,\cdots,x_{n-1}\in M\setminus Z\\ n\ge 1, x_0:=x,x_n:=y}}\sum_i \frac{r}{K(\dist_M(x_i,x_{i+1})/r)}
\ge
\tilde{C}\,\dist_M(x,y).    
\end{equation}

Fix $\varepsilon>0$. Choose $t>0$ such that
\[
\frac1{tK(t)}<\tilde{C}+\varepsilon.
\]
Choose a smooth path from $x$ to $y$ of length $L<\dist_M(x,y)+\varepsilon$. Partition it into $N$ equal segments of length $\ell=L/N$, where
\[
N=\left\lfloor \frac{L}{rt}\right\rfloor+1
\]
with $\lfloor \frac{L}{rt}\rfloor$ being the largest integer smaller than or equal to $\frac{L}{rt}$. 
Then $\ell< rt$. Let the partition points be $x_0,x_1,\dots,x_N$ with $x_0:=x$ and $x_N:=y$. Then $\dist_M(x_i,x_{i+1})\le \ell< rt$.  For any $Z$ with $\mu(Z)=0$, if there exists $x_i\in\{x_1,\cdots,x_{N-1}\}$ lies in $Z$, we can take a perturbation $x_i^Z\in M\setminus Z$ with  $\dist_M(x_i,x_i^Z)<(rt-\ell)/2$  and reset $x_i=x_i^Z$. 
Then, for the modified points $x_0,\dots,x_{N}$ with $x_1,\cdots,x_{N-1}\in M\setminus Z$,  we have  $\dist_M(x_i,x_{i+1})<rt$ for any $i=0,1,\cdots,N-1$. Since $K$ is non-increasing, there holds
\[
K\!\left(\frac{\dist_M(x_i,x_{i+1})}{r}\right)
\ge
K(t).
\]
Thus
\[
\sum_{i=0}^{N-1}\frac{r}{K(\dist_M(x_i,x_{i+1})/r)}
\le
N\frac{r}{K(t)}=\frac{tNr}{tK(t)}
\le
\frac{L+rt}{tK(t)}
\le
(L+rt)(\tilde{C}+\varepsilon).
\]
Hence, 
$$\sup_{Z\subset M:\mu(Z)=0}\ \inf\limits_{\substack{x_1,x_2,\cdots,x_{n-1}\in M\setminus Z\\ n\ge 1, x_0:=x,x_n:=y}}\sum_i \frac{r}{K(\dist_M(x_i,x_{i+1})/r)}\le (L+rt)(\tilde{C}+\varepsilon) \le (\dist_M(x,y)+\varepsilon+rt)(\tilde{C}+\varepsilon) $$
and therefore
\[
\limsup_{r\to0}\sup_{Z\subset M:\mu(Z)=0}\ \inf\limits_{\substack{x_1,x_2,\cdots,x_{n-1}\in M\setminus Z\\ n\ge 1, x_0:=x,x_n:=y}}\sum_i \frac{r}{K(\dist_M(x_i,x_{i+1})/r)}
\le (\dist_M(x,y)+\varepsilon)(\tilde{C}+\varepsilon).
\]
Letting $\varepsilon\to0$, we get
\[
\limsup_{r\to0}\sup_{Z\subset M:\mu(Z)=0}\ \inf\limits_{\substack{x_1,x_2,\cdots,x_{n-1}\in M\setminus Z\\ n\ge 1, x_0:=x,x_n:=y}}\sum_i \frac{r}{K(\dist_M(x_i,x_{i+1})/r)}
\le
\tilde{C}\,\dist_M(x,y).
\]
Combining this with \eqref{eq:liminf>=dist} gives
\[
\lim_{r\to0}\sup_{Z\subset M:\mu(Z)=0}\inf\limits_{\substack{x_1,x_2,\cdots,x_{n-1}\in M\setminus Z\\ n\ge 1, x_0:=x,x_n:=y}}
\sum_{i=0}^{n-1}\frac{r}{K(\dist_M(x_i,x_{i+1})/r)}
=\tilde{C}\,\dist_M(x,y),
\]
which concludes the proof.
\end{proof}
\begin{proof}[Proof of Theorem \ref{th:gamma-function}]
We have proved the pointwise convergence of $r^{-p}\mathcal{R}_{p,W_r}$ to $C_{p,k}C_{0,k}^{-1}\mathcal{R}_{p,M}$ as $r\to0^+$; see Theorem  \ref{th:W_r-M}. 
Moreover, similar to Proposition \ref{prop:gamma-Rp}, it can be further proved that $r^{-p}\mathcal{R}_{p,W_r}$ Gamma-converges to $C_{p,k}C_{0,k}^{-1}\mathcal{R}_{p,M}$. 
Let $$\Phi_{p,W_r}(f)=\frac{\int_{M}\int_{M}|f(x)-f(y)|^pW_r(x,y)dydx}{r^p\inf\limits_{c\in\R}\int_{M}|f(x)-c|^p\int_{M}W_r(x,y)dydx}\text{ and } \Phi_{p,M}(f)=\frac{\int_{M}|\nabla f(x)|^pdx}{\inf\limits_{t\in\R}\int_{M}|f(x)-t|^pdx}.$$
It can be readily seen from \eqref{eq:cRpWr} that $$c(\mathcal{R}_{p,W_r})=r^p\inf_{f\ne \text{ const}}\Phi_{p,W_r}(f) \text{ and }c(\mathcal{R}_{p,M})=\inf_{f\ne \text{ const}}\Phi_{p,M}(f).$$
And, it can be verified that $\Phi_{p,W_r} $ Gamma-converges to $  C_{p,k}C_{0,k}^{-1}\Phi_{p,M}$ in a similar manner. 

Based on the Gamma convergence from $\Phi_{p,W_r} $ to $  C_{p,k}C_{0,k}^{-1}\Phi_{p,M}$ as $k\to\infty$, we obtain 
$$\lim_{r\to0^+}\frac{1}{r^p}c(\mathcal{R}_{p,W_r})=\frac{C_{p,k}}{C_{0,k}} c(\mathcal{R}_{p,M})$$
which implies the desired equality \eqref{eq:gamma-Cpq-Wr-M}. 

Since the constants $C_{0,k}$, $C_{p,k}$ and $C_{q,k}$ depend on the dimension of $M$, next, we analyze the asymptotic behaviour of $\Big(\frac{C_{p,k}}{C_{0,k}}\Big)^{\frac1p}\Big(\frac{C_{q,k}}{C_{0,k}}\Big)^{-\frac1q}$ when the dimension $k$ tends to $\infty$. 
Notice that
\begin{align*}
\int _{\R^k}|x_1|^pK(|x|)dx&=\int_0^\infty\int_0^\infty r^pK(\sqrt{r^2+h^2})h^{k-2}w_{k-2}dhdr
\\&=w_{k-2}\int_0^\infty\int_0^\infty \hat{r}^{\frac p2}\hat{h}^{\frac {k-2}{2}} K(\sqrt{\hat{r}+\hat{h}})\frac{d\hat{h}}{2\sqrt{\hat{h}}}\frac{d\hat{r}}{2\sqrt{\hat{r}}}
\\&=\frac{w_{k-2}}{4} \int_0^\infty\int_0^\infty \hat{r}^{\frac {p-1}{2}}\hat{h}^{\frac {k-3}{2}} K(\sqrt{\hat{r}+\hat{h}})d\hat{h}d\hat{r}
\end{align*}
where $w_{k-2}$ denotes the 
total $(k-2)$-volume of the unit Euclidean sphere of dimension $k-2$. 
Denote by $$F_{p,k}(f)=\int_0^\infty\int_0^\infty r^{\frac{p-1}{2}}h^{\frac{k-3}{2}}f(r+h)drdh$$
and let
$$\widetilde{F}_{p,q,k}(f)=\Big(\frac{F_{p,k}(f)}{F_{0,k}(f)}\Big)^{\frac1p}\Big(\frac{F_{q,k}(f)}{F_{0,k}(f)}\Big)^{-\frac1q}.$$
We can rewrite 
$$\Big(\frac{C_{p,k}}{C_{0,k}}\Big)^{\frac1p}\Big(\frac{C_{q,k}}{C_{0,k}}\Big)^{-\frac1q}=\widetilde{F}_{p,q,k}(f)$$
where we set $f(t)=K(\sqrt{t})$, $\forall t\ge 0$. 
Let 
\begin{equation*}\label{eq:Apq2}
A_{p,q}=\big(\frac{1}{\sqrt{\pi}}\Gamma(\frac{p+1}{2})\big)^{\frac1p}\big(\frac{1}{\sqrt{\pi}}\Gamma(\frac{q+1}{2})\big)^{-\frac1q}=\frac{\Gamma\!\left(\frac{p+1}{2}\right)^{1/p}}
{\Gamma\!\left(\frac{q+1}{2}\right)^{1/q}}  \,(\sqrt{\pi})^{1/q-1/p}  
\end{equation*}

\textbf{Main case}: 
Let $p>q\ge 1$ and let $f\ge 0$ be compactly supported and locally integrable. 
Then
\begin{equation}\label{eq:FpqkApq}
\lim_{k\to\infty}\widetilde F_{p,q,k}(f)
=
A_{p,q}.    
\end{equation}

Proof of the main case: 
Substituting  $s=r+h$ into the expression of $F_{p,k}(f)$, one has
\[F_{p,k}(f)=\int_0^\infty f(s)\left(\int_0^s r^{\frac{p-1}{2}}(s-r)^{\frac{k-3}{2}}\,dr\right)ds=\upbeta\!\left(\frac{p+1}{2},\frac{k-1}{2}\right)\int_0^\infty f(s)s^{\frac{p+k-2}{2}}\,ds\]
where $\upbeta(\cdot,\cdot)$ denotes the standard Beta function. 
Similarly,
\[F_{0,k}(f)=\upbeta\!\left(\frac12,\frac{k-1}{2}\right)\int_0^\infty f(s)s^{\frac{k-2}{2}}\,ds,\]
and
\[
F_{q,k}(f)
=
\upbeta\!\left(\frac{q+1}{2},\frac{k-1}{2}\right)
\int_0^\infty f(s)s^{\frac{q+k-2}{2}}\,ds.
\]
Define
\[
I_r(k):=\int_0^\infty f(s)s^{\frac{r+k-2}{2}}\,ds,\qquad \text{where } r=0,p,q.
\]
Then
\begin{equation}
\label{eq:Ftilde-pqk}    
\widetilde F_{p,q,k}(f)=
\frac{\upbeta\!\left(\frac{p+1}{2},\frac{k-1}{2}\right)^{1/p}\upbeta\!\left(\frac12,\frac{k-1}{2}\right)^{1/q-1/p}}{\upbeta\!\left(\frac{q+1}{2},\frac{k-1}{2}\right)^{1/q}}
\cdot
\frac{I_p(k)^{1/p}\,I_0(k)^{1/q-1/p}}{I_q(k)^{1/q}}.
\end{equation}
For $a>0$, Stirling's asymptotic formula gives
\[
\upbeta\!\left(a,\frac{k-1}{2}\right)
\asymp
\Gamma(a)\left(\frac{k-1}{2}\right)^{-a},\; \text{ as }k\to\infty.
\]
Thus the Beta‑function-term in \eqref{eq:Ftilde-pqk}  tends to
\begin{equation}\label{eq:beta-limit-simplify}
\frac{\Gamma\!\left(\frac{p+1}{2}\right)^{1/p}\,\Gamma\!\left(\frac12\right)^{1/q-1/p}}{\Gamma\!\left(\frac{q+1}{2}\right)^{1/q}}
=\frac{\Gamma\!\left(\frac{p+1}{2}\right)^{1/p}\,(\sqrt{\pi})^{1/q-1/p}}{\Gamma\!\left(\frac{q+1}{2}\right)^{1/q}}    
\end{equation}
as $k$ goes to $\infty$. 
It remains to show that the $I(k)$‑function-term in \eqref{eq:Ftilde-pqk} satisfies
\begin{equation}\label{eq:remain-Ip0q}
\frac{I_p(k)^{1/p}\,I_0(k)^{1/q-1/p}}{I_q(k)^{1/q}}\longrightarrow 1,\; \text{ as }k\to\infty   \,. 
\end{equation}
Let $M=\inf\{s\ge 0: f(t)=0\text{ for a.e. }t\ge s\}$. Since $f$ is compactly supported and not identically zero, $0<M<\infty$. 
Define a probability measure $\mu_k$ on $[0,M]$ via 
\[
\mu_k[a,b]:=\frac{1}{I_0(k)}\int_a^bf(s)s^{\frac{k-2}{2}}\,ds,\;\; \forall [a,b]\subset[0,M].
\]
Then $I_r(k)=I_0(k)\,\mathbb E_{\mu_k}(s^{r/2})$ and therefore
\begin{equation}\label{eq:I_p0q-reform}
\frac{I_p(k)^{1/p}\,I_0(k)^{1/q-1/p}}{I_q(k)^{1/q}}
=
\frac{\big(\mathbb E_{\mu_k}(s^{p/2})\big)^{1/p}}{\big(\mathbb E_{\mu_k}(s^{q/2})\big)^{1/q}}.    
\end{equation}
We will show $\mu_k\to\delta_M$ weakly, where $\delta_M$ denotes the Dirac measure at $M$, which places all its mass (total measure 
1) on the single point $\{M\}$. For any $\varepsilon>0$,
\[
\mu_k([0,M-\varepsilon])
=
\frac{\int_0^{M-\varepsilon} f(s)s^{(k-2)/2}\,ds}
{\int_0^M f(s)s^{(k-2)/2}\,ds}.
\]
It follows from the definition of $M$ that for any $a\in(M-\varepsilon,M)$, $\int_{(a,M]}f(s)ds>0$. Fix such $a$, and let $C=\int_{(a,M]}f(s)ds>0$. Since $f$ is locally integrable, and compactly supported, we have $f\in L^1[0,M]$.  
Then
\[\int_0^{M-\varepsilon} f(s)s^{(k-2)/2}\,ds\le \|f\|_1(M-\varepsilon)^{(k-2)/2},\]
while for $k\ge 2$,
\[
\int_0^M f(s)s^{(k-2)/2}\,ds
\ge \int_a^M   
f(s)s^{(k-2)/2}\,ds \ge a^{(k-2)/2}\int_a^M   
f(s)ds=Ca^{(k-2)/2}.
\]
Therefore
\[
\mu_k([0,M-\varepsilon])
\le\frac{\|f\|_1}{C}\Big(\frac{M-\varepsilon}{a}\Big)^{\frac{k-2}{2}}\to 0,\;\text{ as }k\to\infty.
\]
Thus $\mu_k$ weakly converges to $\delta_M$, and so for every fixed $r>0$, $\mathbb E_{\mu_k}(s^{r/2})\to M^{r/2}$ as $k\to\infty$. Plugging this into \eqref{eq:I_p0q-reform} gives \eqref{eq:remain-Ip0q}. Combining \eqref{eq:remain-Ip0q} with \eqref{eq:beta-limit-simplify} and \eqref{eq:Ftilde-pqk}  yields the desired limit \eqref{eq:FpqkApq}.

\vspace{0.1cm}

\textbf{Special  non‑compactly supported case}: 
Let $p>q\ge 1$. Let $f:[0,\infty)\to[0,\infty)$ be such that:
\begin{enumerate}[(i).]
\item $f(t)>0$  for all sufficiently large $t>0$;
\item $\displaystyle \int_0^\infty f(t)\,t^N\,dt<\infty$ for some $N>0$;
\item The function $h$ defined by $h(u):=-\log f(e^u)$ is $C^2$-smooth on a tail $(N,+\infty)$ and
$\lim_{u\to+\infty} h''(u)=+\infty. $

\end{enumerate}For example, $f(t)=e^{-t^\alpha}$ is a typical function satisfying the conditions (i)--(iii), where $\alpha>0$ is fixed.  

We shall prove the desired formula \eqref{eq:FpqkApq} when using a function $f$ satisfying the conditions (i), (ii) and (iii). 
It suffices to show \eqref{eq:remain-Ip0q}. 
To this end, putting $s=e^u$, on has
\[
I_r(k)=\int_{-\infty}^{\infty} \exp\left(-h(u)+\frac{k+r}{2}u\right)\,du.
\]
Let $u_k$ be the unique solution of $h'(u_k)=k/2$ for sufficiently large $k$ (such $u_k$ exists and is unique for large $k$, because $h'$ is continuous, eventually strictly increasing and tends to infinity). For $r>0$, the saddle point $u_{k,r}$ of the exponent term $-h(u)+\frac{k+r}{2}u$ satisfies
\[
h'(u_{k,r})=\frac{k+r}{2},
\]
and the Taylor expansion gives
\[
u_{k,r}=u_k+\frac{r}{2h''(u_k)}+o\!\left(\frac{1}{h''(u_k)}\right).
\]
By Laplace's method,
\[
\log I_r(k)
=
-h(u_k)+\frac{k+r}{2}u_k - \frac12\log\left(\frac{h''(u_k)}{2\pi}\right)+ \frac{r^2}{8h''(u_k)}+o(1).
\]
Therefore,
\[
\log \frac{I_r(k)}{I_0(k)}
=
\frac{r}{2}u_k + \frac{r^2}{8h''(u_k)}+o(1).
\]
Consequently,
\[
\log \frac{
I_p(k)^{1/p}\,I_0(k)^{1/q-1/p}
}{
I_q(k)^{1/q}
}
=
\frac1p\log\frac{I_p}{I_0}
-\frac1q\log\frac{I_q}{I_0}
=
\frac{p-q}{8}\cdot\frac{1}{h''(u_k)}+o(1).
\]
Since $h''(u_k)\to+\infty$, we get \eqref{eq:remain-Ip0q},  and the desired limit follows. 
\end{proof}

\begin{remark}
The constant $A_{p,q}$ in the proof of Theorem \ref{th:gamma-function} is relevant to the widely used geometric constant 
$$\mathrm{C}_{d,p}:=\int_{\mathbb{S}^{d-1}}|x_1|^pdx=\frac{2\pi^{\frac{d-1}{2}}\Gamma(\frac{p+1}{2})}{\Gamma(\frac{p+d}{2})}$$
via the limit 
$$\lim_{d\to\infty}\Big(\frac{\mathrm{C}_{d,p}}{\mathrm{C}_{d,0}}\Big)^{\frac1p}\Big(\frac{\mathrm{C}_{d,q}}{\mathrm{C}_{d,0}}\Big)^{-\frac1q} =A_{p,q}\,.$$
\end{remark}

\begin{remark}
In the proof of Theorem \ref{th:gamma-function},  we can 
obtain the inequality
$$\liminf_{k\to\infty}\widetilde F_{p,q,k}(f)\ge A_{p,q}$$
for more general $f\ge 0$, 
where $p\ge q\ge 1$. However, this inequality may be strict. For example, taking $f=\sum_{i=1}^\infty c_i1_{[b_{i-1},b_i]}$ with $b_0=0$, $b_i=e^{2i}$ and $c_i-c_{i+1}=e^{-i^2}$, $i\ge 1$, one can check that $\lim_{k\to\infty}\widetilde F_{p,q,k}(f)=e^{(p-q)/4}A_{p,q}>A_{p,q}$, where $A_{p,q}$ is defined in \eqref{eq:Apq}.  
\end{remark}

\section{Geometry and combinatorics on graphons}
\label{sec:geometry-combi}

This part follows a line of research within the theory of graph limits 
----- 
extending important concepts in graph theory to graph limit theory. 
In this section, we focus on graphons on general finite measure spaces, but our strategy also applies in principle to other classes of graph limits. 
The topological and geometric properties of the domain space $J$ are related to combinatorial properties of the graphon $W$, see \cite{Backhausz,LovaszSzegedy10,JansonOlhede,GlebovKralVolec,GlebovKlimosovaKral}. In Section \ref{sec:hidden-limit}, we consistently set $J$ as $M$ by default; however, in  Section \ref{sec:geometry-combi}, we first consider a general finite measure space $J$ with measure $\mu$, and then, in a specific theorem such as Theorem \ref{thm:geo-Sobolev-conve}, set $J = M$; We will clarify the setup regarding $J$ when necessary. 

In this section, we establish a series of   inequalities  
connecting 
many combinatorial parameters and geometric quantities on graphons. 

\subsection{Sobolev constants on graphon}
In this section, we will study various parameters on graphon in a unified manner  by introducing the following Sobolev constants.

\begin{defn}[Sobolev constants on graphon]
The $k$-th $(p,q)$-Sobolev constant of a graphon $W$ is defined as 
\[\lambda_k(W_{p,q}):= \inf_{
S\subset L^q(J):\mathrm{genus}(S)\ge k}\sup_{f\in S}\frac{\|f\|_{W,p}}{\|f\|_q}\]
where $k$ is a positive integer, $p,q\in[1,\infty]$, $\mathrm{genus}(S)$ denotes the Krasnoselskii genus of a nonempty closed origin-symmetric subset $S$, namely, 
$$\mathrm{genus}(S):=\min\big\{k\in\mathbb{Z}^+:\exists\text{ odd continuous }\eta:S\to \mathbb{R}^k\setminus\{0\}\big\},$$ 
  $\|\cdot\|_{W,p}$ is the graphon-induced $p$-seminorm defined as
\begin{equation}\label{eq:Wp-seminorm}
\|f\|_{W,p}:=\Big(\int_{J\times J}W(x,y)|f(x)-f(y)|^pdxdy\Big)^{\frac1p} ,   
\end{equation}
and $\|f\|_q:=(\int_J|f(x)|^qdx)^{1/q}$ denotes the standard $L^q$-norm of $f$. 
For the case $k=2$, we simply call $\lambda_2(W_{p,q})$ the $(p,q)$-Sobolev constant on graphon $W$, which also satisfies the following equality (see \cite{zhang/dual}):
$$\lambda_2(W_{p,q})=\inf_{f\text{ nonconstant}}\frac{\|f\|_{W,p}}{\inf\limits_{c\in\R}\|f-c\|_q}\,.$$
\end{defn}
Basic knowledge in nonlinear analysis implies that $\lambda_k(W_{p,q})$ is a critical value of the \emph{$(p,q)$-Rayleigh quotient} $\|\cdot\|_{W,p}/\|\cdot\|_q$ associated with $W$. 
\begin{example}Here are some special cases:

\begin{itemize}
\item When $k=p=q=2$, the $(2,2)$-Sobolev constant $\lambda_2(W_{2,2})$ equals the square root of the second eigenvalue of graphon Laplacian. 
\item 
When $k=2$, $p=1$ and $q=\infty$, the $(1,\infty)$-Sobolev constant $\lambda_2(W_{1,\infty})$ equals    $$\inf\limits_{A\subset J:\mu(A)\mu(J\setminus A)>0}2\int_{A\times A^c}W(x,y)dxdy$$ which is indeed the ``mincut'' of $W$. 
\end{itemize}
\end{example}

The seminorm in \eqref{eq:Wp-seminorm} can be naturally extended to the setting of signed graphon, where the concept of signed graphon was introduced by Lov\'asz \cite{Lovasz-signed}. 
For a signed graphon $W$,  we define 
$$\|f\|_{W,p}:=\Big(\int_{J\times J}|W(x,y)|\,\big|f(x)-\mathrm{sgn}(W(x,y))f(y)\big|^pdxdy\Big)^{\frac1p} $$
which serves as a straightforward  generalization of \eqref{eq:Wp-seminorm}, thereby the $k$-th $(p,q)$-Sobolev constant for the signed graphon $W$ can be defined in the same way. 
Sometimes, we also use the 
{\bf maximum $(p,q)$-Sobolev parameter} on a (signed-)graphon defined as  
\[\lambda_{\sup}(W_{p,q}):=\sup_{f\in L^q(J)\setminus\{0\}}\frac{\|f\|_{W,p}}{\|f\|_q}.\]
In fact, $\lambda_s(W_{p,q})$ can represent different types of parameters when varying $s$, $p$ and $q$. 
\begin{theorem}
\label{pro:description}
For any graphon $W$, we have the following properties, in which the precise meanings of the concepts of maxcut value, diameter, $p$-capacity and conductance are presented in Sections \ref{sec:maxcut}--\ref{sec:conductance}, respectively.  
\begin{enumerate}[(i)]
\item $(\lambda_2(W_{p,\infty}))^p=2^p\inf\limits_{A\cap B=\varnothing:\mu(A)\mu(B)>0}\mathrm{Cap}_p(A,B)$, where $$\mathrm{Cap}_p(A,B):=\inf_{\substack{f=1\text{ on } A\\ f=0\text{ on } B }}\|f\|_{W,p}^p$$
is a graphon counterpart of the geometric \textbf{$p$-capacity}, 
 $A\subset J$, $B\subset J$ and $A\cap B=\varnothing$.
 
\item $\lambda_{\sup}(W_{p,\infty})$ equals the \textbf{maxcut} value of $W$ multiplied by $2^{1+\frac{1}{p}}$, where $1\le p<\infty$. 
\item  
$\lambda_2(W_{\infty,\infty})$ equals the inverse of the half \textbf{diameter} of $W$, i.e., $\lambda_2(W_{\infty,\infty})=2/\mathrm{diam}(W)$. 
\item 
$\lambda_2(W_{1,1})$ equals the  \textbf{conductance
} of $W$.
\end{enumerate}
\end{theorem}


\begin{prop}\label{pro:description-signed}
For any signed graphon $W$, we have the following properties, where we refer to Section \ref{sec:conductance} for the concepts of bipartiteness ratio and signed conductance. 
\begin{enumerate}[(i)]
    \item If $W\le 0$, then $\lambda_1(W_{1,1})$ is equal to \textbf{bipartiteness ratio} of the graphon $-W$. 
 \item  $\lambda_1(W_{1,1})$ equals the \textbf{first conductance} of the signed graphon $W$. 
\item If $W\le 0$, then $\lambda_{1}(W_{\infty,\infty})$ equals $2/\mathrm{esssup}_{x\in J}l(x)$, where $l(x)$ denotes the infimum of lengths of  
old-cycles passing $x$. 
Here an old-cycle passing $x$ is a sequence of pairwise distinct points $(x_i)_{i=1}^{2k+1}$ with $x\in \{x_i\}_{i=1}^{2k+1}\subset J$ and $k\ge 1$, and the length of $(x_i)_{i=1}^{2k+1}$ is defined by $\frac{1}{W(x_1,x_{2k+1})}+\sum_{i=1}^{2k}\frac{1}{W(x_i,x_{i+1})}$. 
    
\end{enumerate}
\end{prop}

The proofs of Theorem \ref{pro:description} and Proposition \ref{pro:description-signed}, along with more detailed explanations, will be presented in Sections \ref{sec:maxcut}, \ref{sec:sphere-radius}, \ref{sec:capacity} and \ref{sec:conductance}. 
\subsection{Monotonicity theorems relating quantities on graphons}
Next, we shall give a family of inequalities to relate all the parameters above. 
Given $p,q\ge1$, $t\ge1$ with $\frac{p}{q}(1-\frac1t)<1$, let  
$$s_t=\big(1-\frac{p}{q}(1-\frac1t)\big)^{-1} \,.
$$ 
It is easy to see that $s_t<t$ if $p<q$, and $s_t>t$ if $p>q$. 
Moreover, $s_t$ is 
increasing 
with respect to $t$. Such $s_t$ is the unique $s$ that satisfies $ps^*=qt^*$, where $s^*$ and $t^*$ denote the H\"older conjugates of $s$ and $t$, respectively.  

\begin{theorem}\label{thm:monotonicity}
Given a signed graphon $(J,W)$, for any  $k=1,2,\cdots$, the $k$-th $(p,q)$-Sobolev constant $\lambda_k(W_{p,q})$ is continuous with respect to $(p,q)\in[1,\infty]\times [1,\infty]$. 
Moreover,  for any $t\ge 1$ with $\frac{p}{q}(1-\frac1t)<1$, 
\begin{equation}\label{eq:p,q-monotonicity1}
2^{1-t}\lambda_k(W_{pt,qt})^t\le  \lambda_k(W_{p,q})\le tC^{\frac{1}{q}(1-\frac1t)}\lambda_k(W_{ps_t,qt}) 
 \end{equation} 
where $C$ is the maximum degree of $W$.  
In addition, $\|W\|_1 ^{-\frac1p}\lambda_k(W_{p,q})$ increases with respect to $p$, and $\mu(J)^{\frac1q}\lambda_k(W_{p,q})$ decreases with respect to $q$. 
All of the above properties still hold when we use $\lambda_{\sup}(W_{p,q})$  instead of $\lambda_k(W_{p,q})$.
\end{theorem}

By the monotonicity inequalities presented in Theorem \ref{thm:monotonicity}, and according to the explanations of these Sobolev-type constants 
for different $(p,q)$ in Theorem \ref{pro:description} and Proposition \ref{pro:description-signed}, we can get many interesting relations among Sobolev constants, $p$-capacity, 
maxcut, packing radius and Cheeger constants on (signed) graphons. 
For example, taking $p=q=1$, $k=2$ and $t=2$, 
we generalize the Cheeger inequality on graphon \cite{KhetanMJ}. 

It is worth noting that the Sobolev-type constants on graphon satisfy more interrelationships than that on Riemannian manifolds; this is because Theorem \ref{thm:monotonicity} provides more monotonicity inequalities than its geometric counterpart on Riemannian manifolds; Furthermore, the manifold version can be derived from the graphon version—which follows from the following convergence theorem: the $k$-th $(p,q)$-Sobolev constants on graphons converge to that on Riemannian manifolds in the sense of certain scaling limit.
\begin{theorem}\label{thm:geo-Sobolev-conve}
We have the following convergence
$$\lim_{r\to 0}\frac1r\lambda_j\big((W_r)_{p,q}\big)=C_{p,k}^{\frac1p}C_{0,k}^{-\frac1q} \lambda_j(M_{p,q})$$
where $\lambda_j(M_{p,q})$ denotes the  $j$-th $(p,q)$-Sobolev constant on the manifold $M$ of dimension $k$, namely,
\[\lambda_j(M_{p,q}):=\inf_{
S\subset L^q(M)\cap \mathbb{W}^{1,p}(M):\mathrm{genus}(S)\ge j}\sup_{f\in S}\frac{\|\nabla f\|_{p}}{\|f\|_q}\,.\]
\end{theorem}

\begin{remark}
We note that all the geometric quantities appearing in this paper, such as Cheeger constants, sphere packing radii, and $p$-capacities, can be represented using certain $(p,q)$-Sobolev constants. Then, by 
Proposition \ref{prop:sphere-radius-converge}, Theorems \ref{thm:geo-Sobolev-conve} and \ref{thm:Cheeger-set}, these geometric quantities converge in the graphon-to-manifold approximation. 
\end{remark}

\begin{remark}\label{rem:max-min-notconverge}
It is worth noting that the max-min versions of the Sobolev-type constant do not preserve under Gamma convergence, and as a result, the corresponding maximum problem no longer converges.  
This essentially illustrates why some graph parameters, such as the maxcut problem and the signed conductance, have no manifold counterparts; see Remark \ref{rem:maxcut-blowup} and \ref{rem:signed-no-manifold-version}.
\end{remark}

\begin{theorem}\label{thm:monotonicity-manifold}
Given a compact Riemannian manifold $M$, for any  $j\in\{1,2,\cdots\}$, $\lambda_j(M_{p,q})$ is continuous with respect to $(p,q)\in[1,\infty]\times [1,\infty]$. 
Moreover,  for any $t\ge 1$, 
\begin{equation}\label{eq:p,q-monotonicity2}
C_{p,k}^{\frac1p}C_{0,k}^{-\frac1q} \lambda_j(M_{p,q})\le tC_{0,k}^{\frac{1}{q}(1-\frac2t)}C_{ps_t,k}^{\frac{1}{ps_t}}\lambda_j(M_{ps_t,qt}) .
 \end{equation}  
In addition, $(C_{0,k}\vol(M))^{-\frac1p}C_{p,k}^{\frac1p} \lambda_j(M_{p,q})$ is increasing with respect to $p$, and  $\vol(M)^{\frac1q}C_{0,k}^{-\frac1q} \lambda_j(M_{p,q})$ is decreasing with respect to $q$.
\end{theorem}

The proofs of Theorems \ref{thm:monotonicity}, \ref{thm:geo-Sobolev-conve} and \ref{thm:monotonicity-manifold} are presented in Section \ref{sec:monotonicity}.

\subsection{Maxcut problem on graphons}
\label{sec:maxcut}
Let $W:J\times J\to [0,1]$ be a graphon. Consider the optimization problem
$$\text{MaxCut}(W):=\sup_{A\in \mathcal{B}(J)}\int_{A\times A^c} W(x,y)dxdy$$
which can be interpreted as the maxcut problem on $W$.  This quantity has a direct relation with the cut norm; see Proposition \ref{pro:maxcut}. 
Note that the maxcut problem is well-defined in certain measurable combinatorial cases (e.g., graphon), but there is no analog in Riemannian geometry.
\begin{proof}[Proof of  Theorem \ref{pro:description} (ii)]
By the convexity of $f\mapsto \|f\|_{W,p}$, the supremum of $\|\cdot\|_{W,p}$ on the unit $L^\infty$-sphere $\|\cdot\|_\infty= 1$ equals the supremum of $\|\cdot\|_{W,p}$ restricted on the extreme points of the unit $L^\infty$-ball $\|\cdot\|_\infty\le 1$. 
Therefore, 
\begin{align*}
\sup_{f\in L^\infty(J)}\frac{\|f\|_{W,p}}{\|f\|_\infty}&=\sup_{\|f\|_\infty=1}\Big(\int_{J\times J}W(x,y)|f(x)-f(y)|^pdxdy\Big)^{\frac1p}
\\&=\sup_{f\text{ is }\pm 1\text{ valued}}\Big(\int_{J\times J}W(x,y)|f(x)-f(y)|^pdxdy\Big)^{\frac1p}   
\\&=\sup_{A\in \mathcal{B}(J)}\Big(2\int_{A\times A^c} W(x,y)2^pdxdy \Big)^{\frac1p} =2^{1+\frac1p}\mathrm{MaxCut}(W)
\end{align*}
and every $f$ realizing the supremum gives a maxcut $(A,A^c)$ with $A=f^{-1}(1)$. 
\end{proof}
\begin{prop}\label{pro:maxcut}
For any (bounded) graphon $(J,W)$ then there holds
\begin{equation}\label{eq:maxcut=1/infty}
\mathrm{MaxCut}(W)=\sup_{f\in L^\infty(J)} \frac{\int_{J\times J}W(x,y)|f(x)-f(y)|dxdy}{4\|f\|_\infty}\,,    
\end{equation}
and, if the finite measure space $J$ is non-atomic, then
$$\frac14\|W\|_\square\le \mathrm{MaxCut}(W)\le \|W\|_\square$$
\end{prop}
\begin{proof}The  direction $\mathrm{MaxCut}(W)\le \|W\|_\square$ is trivial. We only focus on the other direction $\mathrm{MaxCut}(W)\ge \|W\|_\square/4$. 

For any $\varepsilon>0$, there exists $A\in\B(J)$ such that $\mathrm{MaxCut}(W)\ge \int_{A\times A^c} W(x,y)dxdy>\mathrm{MaxCut}(W)-\varepsilon$. 
Let 
$$A_{\varepsilon}:=\Big\{x\in A:\int_{A^c}W(x,y)dy< \int_{A}W(x,y)dy-2\sqrt{\|W\|_\infty\varepsilon}\Big\}.$$
If $\mu(A_{\varepsilon})\ge\sqrt{\varepsilon/\|W\|_\infty}$, then according to Sierpiński's theorem on non-atomic measures, one can take $B\subset A_{\varepsilon}$ with $\mu(B)=\sqrt{\varepsilon/\|W\|_\infty}$, and then we have
\begin{align*}
\int_{B}\int_{A\setminus B}W(x,y)dydx&=\int_{B}\int_{A}W(x,y)dydx-\int_{B}\int_{B}W(x,y)dydx\\&>\int_{B}\Big(\int_{A^c}W(x,y)dy+2\sqrt{\|W\|_\infty\varepsilon}\Big)dx-(\mu(B))^2\|W\|_\infty
\\&=\int_{B}\int_{A^c}W(x,y)dydx+\mu(B)\big(2\sqrt{\|W\|_\infty\varepsilon}-\mu(B)\|W\|_\infty\big)
\\&=\int_{B}\int_{A^c}W(x,y)dydx + \varepsilon
\end{align*}
which implies 
$$ \mathrm{MaxCut}(W)\le \int_{A}\int_{A^c}W(x,y)dydx+\varepsilon< \int_{(A\setminus B)^c}\int_{A\setminus B}W(x,y)dydx\le \mathrm{MaxCut}(W) $$
a contradiction. Therefore, $\mu(A_{\varepsilon})<\sqrt{\varepsilon/\|W\|_\infty}$, and this implies 
$$\int_{A\setminus A_{\varepsilon}}\Big(\int_{A^c}W(x,y)dy- \int_{A}W(x,y)dy\Big)dx\ge-2\mu(A\setminus A_{\varepsilon})\sqrt{\|W\|_\infty\varepsilon} \ge -2\mu(J)\sqrt{\|W\|_\infty\varepsilon}$$
and 
$$\int_{ A_{\varepsilon}}\Big(\int_{A^c}W(x,y)dy- \int_{A}W(x,y)dy\Big)dx\ge-\mu(A_{\varepsilon})\|W\|_\infty \mu(A) \ge -\mu(J)\sqrt{\|W\|_\infty\varepsilon}$$
and therefore 
$$\int_{A}\Big(\int_{A^c}W(x,y)dy- \int_{A}W(x,y)dy\Big)dx \ge -3\mu(J)\sqrt{\|W\|_\infty\varepsilon}.$$
Hence, 
$\mathrm{MaxCut}(W)\ge \int_{A\times A}W(x,y)dxdy-3\mu(J)\sqrt{\|W\|_\infty\varepsilon}$, and letting $\varepsilon\to0$ we get $\mathrm{MaxCut}(W)\ge \int_{A\times A}W(x,y)dxdy$. Since the positions of $A$ and $A^c$ are symmetric, we can similarly obtain $\mathrm{MaxCut}(W)\ge \int_{A^c\times A^c}W(x,y)dxdy$. Thus, $\mathrm{MaxCut}(W)\ge \|W\|_\square/4$.

The second assertion, i.e. \eqref{eq:maxcut=1/infty}, is a consequence of Theorem \ref{pro:description} (ii). An alternative simple proof of \eqref{eq:maxcut=1/infty} based on Choquet-type extension theory can be found in \cite{Zhang/extension}.
\end{proof}

\begin{remark}\label{rem:maxcut-blowup}
Part of Proposition \ref{pro:maxcut} generalizes a key property of the greedy algorithm for the maxcut problem on classical graphs to the graphon setting. 
Moreover, by Proposition \ref{pro:maxcut}, $$\mathrm{MaxCut}(W_r)\ge \|W_r\|_\square/4=\|W_r\|_1/4=\|W\|_1/4=\|W\|_\square/4$$ and thus $$r^{-1}\mathrm{MaxCut}(W_r)\to+\infty\text{ as }r\to0^+$$ which reveals a ``blowup  phenomenon''. %
This explains why the maxcut problem has no geometric version.
\end{remark}

From the proof of Proposition \ref{pro:maxcut}, one can see that only the inequality $\mathrm{MaxCut}(W)\ge \|W\|_\square/4$ uses the non-atomic condition of $J$; and in fact, such inequality will not hold if one removes the non-atomic condition (for instance, taking $J=\{1\}$, $\mu\{1\}=1$ and $W(1,1)=1$, then $\|W\|_\square=1$ while $\mathrm{MaxCut}(W)=0$). 

\subsection{Sphere packing radius and diameter of graphons}
\label{sec:sphere-radius}

Given a graphon $W$, let 
$$\mathrm{dist}_W(x,y):=\mathop{\mathrm{essinf}}\limits_{\substack{x_0,x_1,\cdots,x_n\in J\\ n\ge 1, x_0=x,x_n=y}}\sum_{i=0}^{n-1}\frac{1}{W(x_i,x_{i+1})}:=\sup_{Z\subset J\setminus \{x,y\}:\mu(Z)=0}\inf\limits_{\substack{x_0,\cdots,x_{n}\in J\setminus Z\\ n\ge 1, x_0=x,x_n=y}}\sum_{i=0}^{n-1}\frac{1}{W(x_i,x_{i+1})}
$$
denote the ``essential'' shortest path distance between 
$x$ and $y$ in $J$ with respect to $W$. 
Note that this leads to a  combinatorial distance on graphons, which is different from the neighborhood distance \cite{LovaszSzegedy10}.

Given $k\ge2$, the $k$-th packing radius  $\mathrm{Spr}_k(W)$
is defined by 
$$\mathrm{Spr}_k(W):=\inf_{Z\subset J:\mu(Z)=0}\mathrm{sup}\big\{r\ge0:\exists x_1,\cdots,x_k\in J\setminus Z\text{ s.t. }\mathrm{dist}_W(x_i,x_j)\ge 2r,\,\forall i\ne j\big\}. $$
\begin{proof}[Proof of  Theorem \ref{pro:description} (iii)]
By the definition of $\mathrm{Spr}_k(W)$ at $k=2$, one has
$$\mathrm{Spr}_2(W)=\frac12\mathop{\mathrm{esssup}}\limits_{x,y\in J}\mathrm{dist}_W(x,y):=\frac12 \inf_{Z\subset J:\mu(Z)=0}\sup\limits_{x,y\in J\setminus Z}\mathrm{dist}_W(x,y)=:\frac{1}{2}\mathrm{diam}(W)$$
which indicates the half diameter of $W$. 
We then have an 
equality
$$\frac{1}{\mathrm{Spr}_2(W)}=\inf_{f\ne 0:\,\mathrm{esssup}\,f\,+\,\mathrm{essinf}\,f=0}\mathcal{R}_{\infty,W}(f)=:\lambda_2(W_{\infty,\infty})$$ which proves $\lambda_2(W_{\infty,\infty})=2/\mathrm{diam}(W)$, 
where $$\mathcal{R}_{\infty,W}(f):=\frac{\mathop{\mathrm{esssup}}\limits_{x,y\in J}W(x,y)|f(x)-f(y)|}{\|f\|_\infty}$$
represents the $\infty$-Rayleigh quotient of $f$ with respect to graphon $W$. 
\end{proof}

We now have $1/\mathrm{Spr}_2(W)=\lambda_2(W_{\infty,\infty})$, and for a general $k\ge2$, we have the following result.
\begin{prop}\label{prop:k-packing-radius>=} For $k=2,3,\cdots$, we have
$$\frac{1}{\mathrm{Spr}_k(W)}\ge \lambda_k(W_{\infty,\infty}):=
\inf_{S\subset L^\infty(J):\mathrm{genus}(S)\ge k}\sup_{f\in S}\mathcal{R}_{\infty,W}(f),
$$
where $\lambda_k(W_{\infty,\infty})$ indicates the $k$-th min-max critical value of the $\infty$-Rayleigh quotient $\mathcal{R}_{\infty,W}(\cdot)$.   
\end{prop}
\begin{proof}
For any $\varepsilon>0$, for any $Z\subset J$ with $\mu(Z)=0$, there exist $x_1,\cdots,x_k\in J\setminus Z$ such that $\mathrm{dist}_W(x_i,x_j)\ge 2\,\mathrm{Spr}_k(W)-2\varepsilon$, $\forall i\ne j$. 
For $i=1,2,\cdots,k$, let $f_i:J\to[0,\infty)$ be defined by
$$f_i(x)=\begin{cases}
\mathrm{Spr}_k(W)-\varepsilon-\mathrm{dist}_W(x_i,x),&\text{ if }\mathrm{dist}_W(x_i,x)\le \mathrm{Spr}_k(W)-\varepsilon\\
0,&\text{ otherwise},   
\end{cases}
$$
and let $S=\mathrm{span}(f_1,\cdots,f_k)\setminus\{0\}\subset L^\infty(J)$. Then $\mathrm{genus}(S)\ge k$. Since every $f\in S$ can be uniquely expressed as $f=\sum_{i=1}^kt_if_i$ with $(t_1,\cdots,t_k)\in\R^k\setminus\mathbf{0}$, we have 
\begin{equation}\label{eq:lambda_kinftyinfty<=}
\lambda_k(W_{\infty,\infty})\le \sup_{f\in S}\mathcal{R}_{\infty,W}(f)= \sup_{(t_1,\cdots,t_k)\in\R^k\setminus\mathbf{0}}\frac{\mathop{\mathrm{esssup}}\limits_{x,y\in J}W(x,y)\big|\sum_{i=1}^kt_i(f_i(x)-f_i(y))\big|}{\mathop{\mathrm{esssup}}\limits_{x\in J}|\sum_{i=1}^kt_if_i(x)|}
\end{equation}

\noindent{\bf Claim}. $\big|\sum_{i=1}^kt_i(f_i(x)-f_i(y))\big|\le \max\limits_{1\le i\le k}|t_i|\,\mathrm{dist}_W(x,y)$. 

Proof of Claim: Without loss of generality, suppose that $\mathrm{dist}_W(x_1,x)\le \mathrm{Spr}_k(W)-\varepsilon$. Then, $$\mathrm{dist}_W(x_i,x)> \mathrm{Spr}_k(W)-\varepsilon$$ for any $i\ne 1$, and thus $f(x)=t_1(\mathrm{Spr}_k(W)-\varepsilon-\mathrm{dist}_W(x_1,x))$, where $f:=\sum_{i=1}^kt_if_i$. Regarding the position of $y\in J$, there are three cases:

Case 1.  $\mathrm{dist}_W(x_1,y)\le \mathrm{Spr}_k(W)-\varepsilon$.

In this case, we similarly have $f(y)=t_1(\mathrm{Spr}_k(W)-\varepsilon-\mathrm{dist}_W(x_1,y))$. Then, $|f(x)-f(y)|=|t_1|\,|\mathrm{dist}_W(x_1,x)-\mathrm{dist}_W(x_1,y)|\le \max\limits_{1\le i\le k}|t_i|\,\mathrm{dist}_W(x,y)$.

Case 2.  $\mathrm{dist}_W(x_i,y)\le \mathrm{Spr}_k(W)-\varepsilon$ for some $i\ne 1$.

In this case, without loss of generality, we assume that $i=2$, and then we simply have $f(y)=t_2(\mathrm{Spr}_k(W)-\varepsilon-\mathrm{dist}_W(x_2,y))$.  Hence, 
\begin{align*}
|f(x)-f(y)|&\le |t_1|\,(\mathrm{Spr}_k(W)-\varepsilon-\mathrm{dist}_W(x_1,x))+|t_2|\,(\mathrm{Spr}_k(W)-\varepsilon-\mathrm{dist}_W(x_2,y))
\\&\le \max\{|t_1|,|t_2|\}\big(2(\mathrm{Spr}_k(W)-\varepsilon)-\mathrm{dist}_W(x_1,x)-\mathrm{dist}_W(x_2,y)\big)
\\&\le \max\{|t_1|,|t_2|\}\big(2\,\mathrm{Spr}_k(W)-2\varepsilon+\mathrm{dist}_W(x,y)-\mathrm{dist}_W(x_1,x_2)\big)
\le \max\limits_{1\le i\le k}|t_i|\,\mathrm{dist}_W(x,y).
\end{align*}

Case 3. $\mathrm{dist}_W(x_i,y)> \mathrm{Spr}_k(W)-\varepsilon$ for any $i$.

In this case, $f(y)=0$, and $|f(x)-f(y)|= |t_1|\,(\mathrm{Spr}_k(W)-\varepsilon-\mathrm{dist}_W(x_1,x))\le |t_1|\,(\mathrm{dist}_W(x_1,y)-\mathrm{dist}_W(x_1,x))\le \max\limits_{1\le i\le k}|t_i|\,\mathrm{dist}_W(x,y)$.  

This concludes the proof of the claim. 

\vspace{0.1cm} 

Back to the proof of Proposition \ref{prop:k-packing-radius>=}. Since $f_1,\cdots,f_k$ have pairwise disjoint supports, it is clear that   $\mathop{\mathrm{esssup}}\limits_{x\in J}|\sum_{i=1}^kt_if_i(x)|=\max\limits_{1\le i\le k}|t_i|\|f_i\|_\infty=\max\limits_{1\le i\le k}|t_i|\,(\mathrm{Spr}_k(W)-\varepsilon)$. Moreover, the definition of the essential shortest path distance implies $W(x,y)\mathrm{dist}_W(x,y)\le 1$.  
Combining these with the above claim and \eqref{eq:lambda_kinftyinfty<=} yield 
\[\lambda_k(W_{\infty,\infty})\le \sup_{(t_1,\cdots,t_k)\in\R^k\setminus\mathbf{0}}\frac{\mathop{\mathrm{esssup}}\limits_{x,y\in J}W(x,y)\max\limits_{1\le i\le k}|t_i|\,\mathrm{dist}_W(x,y)}{\max\limits_{1\le i\le k}|t_i|\,(\mathrm{Spr}_k(W)-\varepsilon)} 
\le \frac{1}{\mathrm{Spr}_k(W)-\varepsilon}\]

Letting $\varepsilon\to0^+$ on the right-hand side of the inequality above, we finally obtain the desired inequality in Proposition \ref{prop:k-packing-radius>=}.
\end{proof}

\begin{prop}\label{prop:sphere-radius-converge}
Let $M$ be a compact Riemannian manifold of dimension $k$, and let $W_r:M\times M\to [0,1]$ be the graphon with the same assumptions as shown in Subsections \ref{sec:graph-graphon} and \ref{sec:graphon-manifold}, and with the additional assumption on the kernel function that $K(t)=1$ if $0\le t\le 1$ and  $K(t)=0$ if $t>1$. Then, for any $j\ge 1$, we have
$$ \lim_{r\to0^+}\frac{\mathrm{Spr}_j(W_r)}{r^{k-1}}=\mathrm{Spr}_j(M)\text{ and }\lim_{r\to0^+}r^{k-1}\lambda_j\big((W_r)_{\infty,\infty}\big)=\lambda_j(M_{\infty,\infty}). $$
\end{prop}
\begin{proof}
The assumption of 
$K$ implies that the constant $\tilde{C}$ in Theorem \ref{thm:metric-converge} equals 1, and then by Theorem \ref{thm:metric-converge}, we have $\lim\limits_{r\to0^+}r^{1-k}\mathrm{dist}_{W_r}(x,y)=\mathrm{dist}_{M}(x,y)$, from which the limit $\lim\limits_{r\to0^+}r^{1-k}\mathrm{Spr}_j(W_r)=\mathrm{Spr}_j(M)$ follows.  

Based on the Gamma-convergence $\Gamma\text{-}\lim\limits_{r\to0^+}r^{k-1}\mathcal{R}_{\infty,W_r}= \mathcal{R}_{\infty,M}$  
(essentially established by Roith and Bungert  \cite{RoithBungert23}), and according to the discussion in Appendix \ref{sec:generalized-Gamma}, we obtain the limit $\lim_{r\to0^+}r^{k-1}\lambda_j\big((W_r)_{\infty,\infty}\big)=\lambda_j(M_{\infty,\infty})$.
\end{proof}

\begin{remark}
Combining the above two propositions, we have $1/\mathrm{Spr}_j(M)\ge \lambda_j(M_{\infty,\infty})$ which generalizes the inequality in the manifold setting.    
\end{remark}

The proof of Proposition \ref{pro:description-signed} (iii) is elementary and is very similar to that of Theorem \ref{pro:description} (iii), and thus we omit it.

\subsection{Graphon $p$-capacity}
\label{sec:capacity}

The $p$-capacity of a disjoint nonempty pair $(A,B)$ on the graphon $W$ is defined as 
$$\mathrm{Cap}_p(A,B):=\inf_{\substack{f=1\text{ on } A\\ f=0\text{ on } B }}\|f\|_{W,p}^p,$$
where $A,B\in\B(J)$ satisfy $\mu(A)\mu(B)>0$.

Denote by $C_{A,B}:=\int_{J\times J\setminus\big((A\times A)\cup(B\times B)\big)}W(x,y)dxdy$. Then by H\"older's inequality (or the power mean  inequality), 
$$p\mapsto \Big(\frac{\mathrm{Cap}_p(A,B)}{|C_{A,B}|}\Big)^{\frac1p}$$
is an increasing function of $p$. This can be seen as a graphon counterpart of Eq.(5.2) on Euclidean domains in Lindqvist's paper \cite{Lindqvist93}. And this is also compatible with the increasing property of $$p\mapsto \Big(\frac{\lambda_2^p(W_{p,\infty})}{\|W\|_1}\Big)^{\frac1p}$$ which is included in Theorem \ref{thm:monotonicity}. 

Following the $p$-resistance on graphs \cite{Saito23}, we introduce the \emph{graphon $p$-resistance} defined by 
$$r_{W,p}(x,y):=\frac{1}{\min\limits_{f\text{ s.t. }f(x)-f(y)=1}\|f\|_{W,p}^p}$$
where $\{x\}$ and  $\{y\}$ are two atoms of the measure space $J$. Such $p$-resistance on graphon can be reformulated in terms of the $p$-capacity via the equality
$$ r_{W,p}(x,y)= \frac{1}{\mathrm{Cap}_p(\{x\},\{y\})}. $$

\begin{proof}[Proof of  Theorem \ref{pro:description} (i)]We shall prove 
$$ \inf\limits_{A\cap B=\varnothing:\mu(A)\mu(B)>0}\mathrm{Cap}_p(A,B)=\frac{1}{2^p}\big(\lambda_2(W_{p,\infty})\big)^p. $$

It is not difficult to see that we can further assume $0\le f\le 1$ in the definition of $\mathrm{Cap}_p(A,B)$. In fact, for any $f$ with $f=1$ on $A$ and $0$ on $B$, we take $$\tilde{f}(x)=\begin{cases}
1,&\text{ if }f(x)>1,\\
f(x),&\text{ if }0\le f(x)\le1,\\
0,&\text{ if }f(x)<0.\\
\end{cases}$$
Then $\big|\tilde{f}(x)-\tilde{f}(y)\big|\le |f(x)-f(y)|$ for any $x,y\in V$, which means that $\|\tilde{f}\|_{W,p}\le \|f\|_{W,p}$.


Consider $$\lambda_2^p(W_{p,\infty})=\inf_{f\text{ nonconstant}}\frac{\|f\|_{W,p}^p}{\min\limits_{t\in\R}\|f-t\|_\infty^p}.$$
Note that   $\min\limits_{t\in\R}\|f-t\|_\infty=\frac12(\mathrm{esssup}\,f-\mathrm{essinf}\,f)$ and $\lambda_2^p(W_{p,\infty})=\inf\limits_{\mathrm{esssup}\,f-\mathrm{essinf}\,f=2}\|f\|_{W,p}^p$, and without loss of generality, we can assume $\mathrm{esssup}\,f=1=-\,\mathrm{essinf}\,f$. We shall prove
\begin{equation}\label{eq:lambda2Wpinfty}
\lambda_2^p(W_{p,\infty})=\inf_{A\cap B=\varnothing:\mu(A)\mu(B)>0}\inf_{\substack{f=1\text{ on } A\\ f=-1\text{ on } B \\-1\le f\le1}}\|f\|_{W,p}^p.    
\end{equation}
To show the equality \eqref{eq:lambda2Wpinfty}, the direction ``$\le$'' is clear, and it suffices to check the other direction  ``$\ge$''. For any $\varepsilon>0$, there exists $f$ with $\mathrm{esssup}\,f=1=-\,\mathrm{essinf}\,f$ and $\lambda_2(W_{p,\infty})\le\|f\|_{W,p}< \lambda_2(W_{p,\infty})+\varepsilon$. Let $f_n$ be defined by
$$f_n(x)= \begin{cases}
1,&\text{ if }f(x)>1-\frac1n,\\
-1,&\text{ if }f(x)<-1+\frac1n,\\
f(x),&\text{ if }|f(x)|\le 1-\frac1n.\\
\end{cases}$$
Then $\|f_n-f\|_\infty\le \frac1n$ and thus $\|f_n-f\|_{W,p}\le \frac{2}{n}\|W\|_1^{\frac1p}$. The right-hand side of \eqref{eq:lambda2Wpinfty} is smaller than or equal to $\|f_n\|_{W,p}^p\le (\|f\|_{W,p}+\frac{2}{n}\|W\|_1^{\frac1p})^p\le(\lambda_2(W_{p,\infty})+\varepsilon+\frac{2}{n}\|W\|_1^{\frac1p})^p$. 
Letting $n\to\infty$ and $\varepsilon\to0$, we obtain that the right-hand side of \eqref{eq:lambda2Wpinfty} is smaller than or equal to $\lambda_2^p(W_{p,\infty})$. This concludes the proof of \eqref{eq:lambda2Wpinfty}. 

Since $\|f+1\|_{W,p}=\|f\|_{W,p}$, we have 
\begin{align*}
\inf_{\substack{f=1\text{ on } A\\ f=-1\text{ on } B \\-1\le f\le1}}\|f\|_{W,p}^p=\inf_{\substack{f=1\text{ on } A\\ f=-1\text{ on } B \\-1\le f\le1}}\|f+1\|_{W,p}^p=\inf_{\substack{g=2\text{ on } A\\ g=0\text{ on } B \\0\le g\le2}}\|g\|_{W,p}^p
=\inf_{\substack{h=1\text{ on } A\\ h=0\text{ on } B \\0\le h\le1}}2^p\|h\|_{W,p}^p=2^p\mathrm{Cap}_p(A,B).
\end{align*}
This completes the proof of  Theorem \ref{pro:description} (i).
\end{proof}



\subsection{Conductance 
on (signed) graphons}
\label{sec:conductance}


\begin{defn}[signed graphon conductance]
For a bounded signed graphon $W:J\times J\to \R$, we define the first signed conductance 
$$h_1(W):=\inf_{\substack{\text{measurable }A,B\subset J\\
A\cap B=\varnothing,\,\mu(A\cup B)>0}} 
\psi(A,B)
$$
and the second signed conductance 
$$h_2(W):= \inf_{\substack{\text{measurable }A_1,B_1,A_2,B_2\subset J\\\mu(A_1\cup B_1)\mu(A_2\cup B_2)>0\\ A_1\cap B_1=A_2\cap B_2=\varnothing}} 
\max\{\psi(A_1,B_1),\psi(A_2,B_2)\} $$
where 
$$\psi(A,B)=
\frac{4\eta(W_+(A\times B))+2\eta(W_-(A\times A))+2\eta(W_-(B\times B))+2\eta((A\cup B)\times (A\cup B)^c)}{\mu(A\cup B)},$$
$W_\pm(S)=\{(x,y)\in S:\pm W(x,y)>0\}$ and   $\eta(S):=\int_S|W(x,y)|dxdy$ for any measurable subset $S\subset J\times J$. 
\end{defn}

If $W\ge 0$ is a graphon, 
it can be shown that 
$$\lambda_2(W_{1,1})=\inf_{0<\mu(A)<\mu(J)}\frac{2\eta(A\times (J\setminus A))}{\min\{\mu(A),\mu(J\setminus A)\}}= h_2(W).$$

If the signed graphon is non-positive, i.e., $W\le 0$ a.e., then $h_1(W)$ also refers to the \emph{bipartiteness ratio} of $-W$.

In general, it can be verified that
$$\lambda_1(W_{1,1})=\inf_{f\in L^\infty(J)\setminus\{0\}}\frac{\int_{J\times J}|W(x,y)|\big|f(x)-\mathrm{sgn}(W(x,y))f(y)\big|dxdy}{\|f\|_1}=
h_1(W).$$ 

The above claims are stated in Theorem \ref{pro:description} (iv), Proposition \ref{pro:description-signed} (i) and (ii), respectively. 


\begin{remark}\label{rem:signed-no-manifold-version}
It should also be noted that the signed graphon has changed sign; in this case, the ‘non-differential’ term $\int_{M\times M}|f(x)+f(y)|dxdy$ does not converge to the gradient term $C\int_M|\nabla f(x)|dx$. 
This fact indicates that the signed conductance and, in particular, the bipartiteness ratio, have no counterparts on manifolds. 
\end{remark}


\begin{defn}[Dirichlet-Cheeger constant of a graphon]\label{defn:Dirichlet-cond}
Let $W:J\times J\to [0,\infty)$ be a bounded graphon, and let $\Omega\in\B$ be such that 
$0<\mu(\Omega)<\mu(J)$. The  Dirichlet-Cheeger constant of $W$ with respect to $\Omega$ is 
$$h_1(W;\Omega):=\inf_{A\in\B_\Omega:\mu(A)>0}\frac{2\eta(A\times (J\setminus A))}{\mu(A)}
$$
where  $\B_\Omega:=\{S\in \B:S\subset \Omega\}$ and $\B$ is the Borel $\sigma$-algebra on $J$. 
\end{defn}

For a bounded domain $\Omega\subset\R^d$ equipped with the Lebesgue measure $\mu:=\vol$, the first Cheeger constant of $\Omega$ is defined as
$$h_1(\Omega)=\inf\limits_{A\subset\Omega}\frac{\mathcal{H}^{d-1}(\partial A)}{\vol(A)}$$
where $\mathcal{H}^{d-1}$ is the $(d-1)$-dimensional Hausdorff measure. 
We call $A$ a \emph{Cheeger set} of $\Omega$ if  $\mathcal{H}^{d-1}(\partial A)/\vol(A)=h_1(\Omega)$. 

\begin{theorem}\label{thm:Cheeger-set}
Suppose that $\Omega\subset\R^d$ is a bounded closed domain with $\mathcal{H}^{d-1}(\partial\Omega)<\infty$. 
Let $U$ be a bounded open neighborhood of $\Omega$. 
Let $W:U\times U\to[0,\infty)$ be defined via $W(x,y)=K(|x-y|)$ where $K$ is defined in Assumption \ref{ass:1}. 
Then, $\lim_{r\to0} r^{-1}h_1(W_r;\Omega)=
C_{1,d}h_1(\Omega)$, where the constant $C_{1,d}$  depending only on the dimension $d$ is defined in \eqref{eq:C_0k,C_pk}. Moreover, suppose that $A_r$ is a Cheeger set corresponding to $h_1(W_r;\Omega)$. Then, $\{A_r\}_{r>0}$ has a  subsequence that converges to a Cheeger set of $\Omega$ (in the sense of Hausdorff convergence). 
\end{theorem}

To show a proof of Theorem \ref{thm:Cheeger-set}, we need to prove some lemmas as a preparatory step.

\begin{lemma}\label{eq:LinftyL1subconverge}
If there exists $C>0$ such that  $\|u_\varepsilon\|_{L^\infty(\mathbb{R}^d)}<C$, $\bigcup\limits_{\varepsilon>0}\mathrm{supp}(u_\varepsilon)$ is bounded, and
$$
\int_{\mathbb{R}^d}\int_{\mathbb{R}^d}\frac1\varepsilon K_\varepsilon(|x-y|)|u_\varepsilon(x)-u_\varepsilon(y)|\,dx\,dy<C,
$$
then $\{u_\varepsilon\}_{\varepsilon>0}$ has a subsequence converging in $L^1(\mathbb{R}^d)$.
\end{lemma}

\begin{proof}
Let $\widetilde{K}(x)=K(|x|)$ for any $x\in\R^d$. 
There exists a nonnegative smooth compactly supported function $\phi\le \widetilde{K}$ with $|\nabla\phi|\le\widetilde{K}$; and without loss of generality, we may assume $\|\phi\|_{L^1}=1$ and $\phi$ is even (i.e.,  $\phi(-y)=\phi(y)$, $\forall y$). Then it is easy to verify that $\phi_\varepsilon\le \widetilde{K}_\varepsilon$ and $|\nabla\phi_\varepsilon|\le\frac1\varepsilon \widetilde{K}_\varepsilon$. Note that
\begin{align*}
\int_{\mathbb{R}^d}|\phi_\varepsilon*u_\varepsilon-u_\varepsilon|\,dx
&=\int_{\mathbb{R}^d}\left|\int_{\mathbb{R}^d}\phi_\varepsilon(y)(u_\varepsilon(x+y)-u_\varepsilon(x))\,dy\right|dx\\
&\le \int_{\mathbb{R}^d}\int_{\mathbb{R}^d} \widetilde{K}_\varepsilon(y)|u_\varepsilon(x+y)-u_\varepsilon(x)|\,dy\,dx<C\varepsilon,
\end{align*}
and
\begin{align*}
\int_{\mathbb{R}^d}|\nabla(\phi_\varepsilon*u_\varepsilon)|\,dx
&=\int_{\mathbb{R}^d}\left|\int_{\mathbb{R}^d}\nabla\phi_\varepsilon(y)(u_\varepsilon(x+y)-u_\varepsilon(x))\,dy\right|dx\\
&\le \frac1\varepsilon\int_{\mathbb{R}^d}\int_{\mathbb{R}^d} \widetilde{K}_\varepsilon(y)|u_\varepsilon(x+y)-u_\varepsilon(x)|\,dy\,dx<C.
\end{align*}
Since $\|u_\varepsilon\|_{L^1}<C\vol(\Omega)$ implies $\|\phi_\varepsilon*u_\varepsilon\|_{L^1}<C(\vol(\Omega)+\varepsilon)$, and $\|\nabla(\phi_\varepsilon*u_\varepsilon)\|_{L^1}<C$, the sequence $\phi_\varepsilon*u_\varepsilon$ has a subsequence $(\phi_{\varepsilon_k}*u_{\varepsilon_k})_k$ that converges weakly-$*$ to a BV function $w$. Then $(u_{\varepsilon_k})_k$ converges in $L^1$ to $w$.
\end{proof}

\begin{lemma}\label{lem:Linfty-compact}
Under the hypotheses of Theorem \ref{thm:Cheeger-set},  
suppose that 
$$\sup_{\varepsilon>0}\|u_\varepsilon\|_{\infty}<\infty\;\text{ and }\;\sup_{\varepsilon>0}\varepsilon^{-1}\|u_\varepsilon\|_{W_{\varepsilon},1}<\infty,$$ 
then $(u_\varepsilon)_{\varepsilon>0}$ is relatively compact in $L^1(\Omega)$.
\end{lemma}

\begin{proof}
Define the extension $\tilde{u}_\varepsilon(x)=\begin{cases}
u_\varepsilon(x),& x\in\Omega,\\
0,&x\notin\Omega.
\end{cases}$
Then $\tilde{u}_\varepsilon\in L^1\cap L^\infty(\mathbb{R}^d)$. We first prove
\begin{equation}\label{eq:OmegaOmegac}
\int_{\mathbb{R}^d\setminus\Omega}\int_{\Omega}\frac1\varepsilon \widetilde{K}_\varepsilon(x-y)|\tilde{u}_\varepsilon(x)-\tilde{u}_\varepsilon(y)|\,dx\,dy<C.
\end{equation}
By the change of variables $y=x+\varepsilon h$, \eqref{eq:OmegaOmegac} is equivalent to
$$
\int_{\mathbb{R}^d}\int_{\Omega_{\varepsilon h}}\frac1\varepsilon\widetilde{K}(h)|\tilde{u}_\varepsilon(x+\varepsilon h)-\tilde{u}_\varepsilon(x)|\,dx\,dh<C,
$$
where $\Omega_{\varepsilon h}=\{x\in\mathbb{R}^d: x\in\Omega,\ x+\varepsilon h\notin \Omega\}$. Note that for any $x\in\Omega_{\varepsilon h}$, the segment $[x,x+\varepsilon h]$ satisfies $[x,x+\varepsilon h]\cap \partial\Omega\neq\varnothing$. Write $x=z-t\varepsilon h$ for some $z\in\partial\Omega$, $t\in[0,1]$. A direct computation gives
$$
\int_{\Omega_{\varepsilon h}}|\tilde{u}_\varepsilon(x+\varepsilon h)-\tilde{u}_\varepsilon(x)|\,dx
\le
|\varepsilon h|\int_{\partial\Omega}\int_{S_{h,z}}|\tilde{u}_\varepsilon(z+(1-t)\varepsilon h)-\tilde{u}_\varepsilon(z-t\varepsilon h)|\,dt\,dz,
$$
where $S_{h,z}=\{t\in[0,1]: z-t\varepsilon h\in\Omega,\, z+(1-t)\varepsilon h\notin\Omega\}$. Hence
\begin{align*}
&\int_{\mathbb{R}^d}\int_{\Omega_{\varepsilon h}}\frac1\varepsilon\widetilde{K}(h)|\tilde{u}_\varepsilon(x+\varepsilon h)-\tilde{u}_\varepsilon(x)|\,dx\,dh\\
\le &\int_{\mathbb{R}^d}\frac1\varepsilon\widetilde{K}(h)|\varepsilon h|\int_{\partial\Omega}\int_{S_{h,z}}|\tilde{u}_\varepsilon(z+(1-t)\varepsilon h)-\tilde{u}_\varepsilon(z-t\varepsilon h)|\,dt\,dz\,dh\\
\le &\int_{\mathbb{R}^d}\widetilde{K}(h)|h|\int_{\partial\Omega} C\,dz\,dh\le C_{1,d}C\,\mathcal{H}^{d-1}(\partial\Omega)
\end{align*}
where $C:=2\sup\limits_{\varepsilon>0}\|u_\varepsilon\|_{\infty}$. 
Therefore, taking $\tilde{C}=2C_{1,d}C\,\mathcal{H}^{d-1}(\partial\Omega)+\sup\limits_{\varepsilon>0}\varepsilon^{-1}\|u_\varepsilon\|_{W_{\varepsilon},1}$, we have
\begin{align*}
&\int_{\mathbb{R}^d}\int_{\mathbb{R}^d}\frac1\varepsilon \widetilde{K}_\varepsilon(x-y)|\tilde{u}_\varepsilon(x)-\tilde{u}_\varepsilon(y)|\,dx\,dy\\
=\;&2\int_{\mathbb{R}^d\setminus\Omega}\int_{\Omega}\frac1\varepsilon \widetilde{K}_\varepsilon(x-y)|\tilde{u}_\varepsilon(x)-\tilde{u}_\varepsilon(y)|\,dx\,dy+\int_{\Omega}\int_{\Omega}\frac1\varepsilon \widetilde{K}_\varepsilon(x-y)|u_\varepsilon(x)-u_\varepsilon(y)|\,dx\,dy\\
=\;&2\int_{\mathbb{R}^d\setminus\Omega}\int_{\Omega}\frac1\varepsilon \widetilde{K}_\varepsilon(x-y)|\tilde{u}_\varepsilon(x)-\tilde{u}_\varepsilon(y)|\,dx\,dy+\varepsilon^{-1}\|u_\varepsilon\|_{W_{\varepsilon},1}
<\tilde{C}.
\end{align*}
By Lemma~\ref{eq:LinftyL1subconverge}, $(u_\varepsilon)_{\varepsilon>0}$ is relatively compact in $L^1(\Omega)$.
\end{proof}

We shall also use the following nonlocal isoperimetric lemma, which follows from the Riesz rearrangement inequality (see \cite[Theorem 3.7]{LiebLoss}) and the usual isoperimetric inequality (see \cite[Chapter 14]{Maggi}). 
\begin{lemma}\label{lem:nonlocal-isoperimetric}
For any $c>0$, there exists a constant $C_d>0$ depending only on the dimension $d$ such that for every measurable set $E\subset\mathbb{R}^d$ and every $0<\delta\le c(\vol(E))^{1/d}$,
\[
\int_E\int_{\mathbb{R}^d\setminus E}\mathbf{1}_{\{|x-y|<\delta\}}\,dx\,dy
\ge C_d\,\delta^{d+1}\big(\vol(E)\big)^{(d-1)/d}.
\]
\end{lemma}

\begin{proof}[Proof of Theorem \ref{thm:Cheeger-set}]
In the context of Theorem \ref{thm:Cheeger-set}, Definition \ref{defn:Dirichlet-cond} for the Dirichlet-Cheeger constant of $W_r$ can be expressed as 
\begin{equation}\label{eq:h_1(W_r;Omega)}
h_1(W_r;\Omega)=\inf_{\substack{\text{Lebesgue measurable }A\subset\Omega\\\vol(A)>0}}\frac{2\int_{A\times (U\setminus A)} W_r(x,y)\,dx\,dy}{\vol(A)}
=\inf_{\substack{A\in\B(\Omega)\\\vol(A)>0}}\frac{\|\mathbf 1_{A}\|_{W_{r},1}}{\|\mathbf 1_{A}\|_1},
\end{equation}
where we implicitly extend $W_r$ by zero outside $U\times U$.  Recalling \eqref{eq:Wp-seminorm}, 
note that 
\[
r^{-1}\|u\|_{W_{r},1}=\frac{2}{r}\int_{U\times U} W_r(x,y)\,|u(x)-u(y)|\,dxdy
\]
is the nonlocal total variation. 
Very similar to the results of Garc\'ia Trillos and Slep\v{c}ev \cite{Garcia-Slepcev} as well as Roith and Bungert \cite{RoithBungert23}, 
one has the Gamma convergence $\Gamma\text{-}\lim_{r\to0}r^{-1}\|\cdot\|_{W_{r},1}=C_{1,d}\|D\cdot\|$, where $\|D\cdot\|$ denotes the total variation functional. %

\medskip\noindent
\textbf{Step 1: The Limsup inequality $\limsup_{r\to0} r^{-1}h_1(W_r;\Omega)\le 
C_{1,d}h_1(\Omega)$.}

Suppose the contrary, that there exist $\delta>0$ and a decreasing sequence $\{r_n\}_{n=1}^\infty$ of positive numbers converging to 0 as $n\to\infty$, such that $r_n^{-1}h_1(W_{r_n};\Omega)> C_{1,d}h_1(\Omega)+\delta$.

Let $B\subset\Omega$ be a set of finite perimeter. Since $\Gamma\text{-}\lim_{r\to0}r^{-1}\|\cdot\|_{W_{r},1}=C_{1,d}\|D\cdot\|$, by the limsup inequality of Gamma‑convergence, there exists $u_r^B\in L^1(U)$ with $u_r^B=0$ on $U\setminus\Omega$ such that $u_r^B\to \mathbf 1_B$ in $L^1$ as $r\to0$, and
\[
\limsup_{r\to0} \frac1r\|u_r^B\|_{W_{r},1} \le C_{1,d}\|D\mathbf 1_B\|= C_{1,d}\,\mathcal H^{d-1}(\partial^* B).
\]
Let $\{B_k\}_{k=1}^\infty$ be a sequence such that  $B_k\subset\Omega$ is of finite perimeter, and
\[
\lim_{k\to\infty}\frac{\mathcal H^{d-1}(\partial^* B_k)}{\vol(B_k)} = h_1(\Omega).
\]
For each $k$, let $u_{r_n}^{(k)}$ be a recovery sequence, i.e., 
$$\limsup_{n\to\infty} \frac{\|u_{r_n}^{(k)}\|_{W_{r_n},1}}{r_n\|u_{r_n}^{(k)}\|_1}\le C_{1,d}\frac{\|D\mathbf 1_{B_k}\|}{\|\mathbf 1_{B_k}\|_1} =C_{1,d}\frac{\mathcal H^{d-1}(\partial^* B_k)}{\vol(B_k)}$$
Then, one can find a subsequence $\{r_{n_k}\}\subset\{r_n\}$ such that 
\begin{equation}\label{eq:Step1-Limsup-1/k}
\frac{\|u_{r_{n_k}}^{(k)}\|_{W_{r_{n_k}},1}}{r_{n_k}\|u_{r_{n_k}}^{(k)}\|_1}\le C_{1,d}\frac{\mathcal H^{d-1}(\partial^* B_k)}{\vol(B_k)}+\frac1k.    
\end{equation}
Rewriting $r_k:=r_{n_k}$ and $u_{r_k}:=u_{r_{n_k}}^{(k)}$, and letting $k\to\infty$, we obtain 
\[
\limsup_{k\to\infty} \frac{\|u_{r_k}\|_{W_{r_k},1}}{{r_k}\|u_{r_k}\|_1} \le C_{1,d} h_1(\Omega).
\]
Using coarea formula for nonlocal energies (or the results on Choquet-type extensions  \cite{Zhang/extension}), one can show that for any such $u_{r_k}$, there exists a superlevel set $\varnothing\ne A_{r_k}\in \B_\Omega$ with $A_{r_k}\subset\mathrm{supp}(u_{r_k})$ such that $\|\mathbf 1_{A_{r_k}}\|_{W_{r_k},1}\big/\|\mathbf 1_{A_{r_k}}\|_1\le \|u_{r_k}\|_{W_{r_k},1}\big/\|u_{r_k}\|_1$. Combining this with the meaning of $h_1(W_{r_k};\Omega)$ in \eqref{eq:h_1(W_r;Omega)}, and the assumption $r_k^{-1}h_1(W_{r_k};\Omega)> C_{1,d}h_1(\Omega)+\delta$, we have 
\[
C_{1,d}h_1(\Omega)+\delta\le \limsup_{k\to\infty} {r_k}^{-1}h_1(W_{r_k};\Omega) \le \limsup_{k\to\infty} 
\frac{\|\mathbf 1_{A_{r_k}}\|_{W_{r_k},1}}{{r_k}\|\mathbf 1_{A_{r_k}}\|_1}
\le \limsup_{k\to\infty} \frac{\|u_{r_k}\|_{W_{r_k},1}}{{r_k}\|u_{r_k}\|_1} 
\]
which contradicts \eqref{eq:Step1-Limsup-1/k}. This completes  Step 1. 

\medskip\noindent
\textbf{Step 2: Compactness of asymptotic minimizers.}

Note that $r^{-1}\|\mathbf 1_\Omega\|_{W_{r},1}\to \|D\mathbf 1_\Omega\|=\mathcal{H}^{d-1}(\partial\Omega)<\infty$ as $r\to0$, and 
$r^{-1}h_1(W_r;\Omega)\le r^{-1}\|\mathbf 1_\Omega\|_{W_{r},1}/\|\mathbf 1_\Omega\|_1$ has an upper bound that is independent of $r>0$. 
Let $A_r\in\B(\Omega)$ be such that
\[
\frac{r^{-1}\|\mathbf 1_{A_r}\|_{W_{r},1}}{\|\mathbf 1_{A_r}\|_1} \le r^{-1}h_1(W_r;\Omega)+\min\{r^2,\,r^{-2}\},
\]
and set $u_r:=\mathbf 1_{A_r}$ (extended by 0 outside $\Omega$). 
Such $\{A_r\}_{r>0}$ can be viewed as asymptotic minimizers of the graphon Cheeger problem on solving $h_1(W_r;\Omega)$. 
Then, it can be checked that 
\begin{equation}\label{eq:Step2-hypotheses-check}
\sup_{r>0} \frac{1}{r}\int_{U\times U} W_r(x,y)|u_r(x)-u_r(y)|\,dx\,dy < \infty\;\text{ and }\;  \|u_r\|_\infty\le1.  
\end{equation}

Since $\Omega$ is compact, and $U\supset\Omega$ is open, one has $\delta:=\inf_{x\in \Omega,y\in \partial U} |x-y|>0$. 
Since $K$ is continuous, nonincreasing, compactly supported, and not identically zero (see Assumption \ref{ass:1} for these conditions), we have $K(0)>0$.
Hence there exist constants $\rho>0$ and $\kappa>0$ such that
\[K(t)\ge \kappa \;\; \text{for all } 0\le t\le \rho.\]
For any $r>0$ and any Lebesgue measurable subsets $E\subset E_1\subset\R^d$, define the \emph{nonlocal perimeter} of $E$ relative to $E_1$ by
\[
P_r^{E_1}(E)=\int_{E\times(E_1\setminus E)} \mathbf{1}_{\{|x-y|<r\}}\,dx\,dy.
\]
Let $B\supset U$ be a large ball with $\vol(B)>2\,\vol(U)$, then for any $E\subset \Omega$, when $0<r<\delta$,  
we have $P_r^{U}(E)=P_r^{B}(E)$. 
Since $\inf_{x\in E,y\in \R^n\setminus B} |x-y|\ge\delta>r$, we further have $P_r^{U}(E)=P_r^{B}(E)=P_r^{\R^n}(E)$. 

If $|x-y|<r\rho$, then $|x-y|/r<\rho$ and hence $K(|x-y|/r)\ge \kappa$. 
Thus,
\begin{equation}\label{eq:uniform-bounded-Step2}
\frac{\|\mathbf 1_{A_r}\|_{W_{r},1}}{r\|\mathbf 1_{A_r}\|_1}=
\frac{r^{-1}}{\vol(A_r)}\int_{A_r}\int_{\mathbb{R}^d\setminus A_r}\frac{1}{r^{d}}K(\frac{|x-y|}{r})\,dxdy 
\ge  \frac{\kappa r^{-d-1}}{\vol(A_r)}
\int_{A_r}\int_{\mathbb{R}^d\setminus A_r}\mathbf{1}_{\{|x-y|<r\rho\}}\,dxdy
\end{equation}
is uniformly bounded for sufficiently small $r>0$. 

We now prove that $\vol(A_r)$ is uniformly bounded from below, for sufficiently small $r>0$. 

To this end, we first show  $\vol(A_r)\ge\frac{\omega_d}{2}(r\rho)^d$,  where $\omega_d$ denote the volume of the unit ball in $\mathbb{R}^d$. 

Suppose the contrary, that  $\vol(A_r)<\frac{\omega_d}{2}(r\rho)^d$.  
Then, for every $x\in A_r$,
\[
\vol(B(x,r\rho)\setminus A_r)\ge \vol(B(x,r\rho))-\vol(A_r)\ge \omega_d(r\rho)^d-\frac{\omega_d}{2}(r\rho)^d=\frac{\omega_d}{2}(r\rho)^d.
\]
Hence
\[
\int_{A_r}\int_{\mathbb{R}^d\setminus A_r}\mathbf{1}_{\{|x-y|<r\rho\}}\,dy\,dx=\int_{A_r}\vol(B(x,r\rho)\setminus A_r)\,dx
\ge \vol(A_r)\frac{\omega_d}{2}(r\rho)^d .
\]
Thus
\[
\frac{r^{-d-1}}{\vol(A_r)}\int_{A_r}\int_{\mathbb{R}^d\setminus A_r}\mathbf{1}_{\{|x-y|<r\rho\}}\,dx dy \ge  \frac{r^{-d-1}}{\vol(A_r)}\cdot \vol(A_r)\frac{\omega_d}{2}(r\rho)^d =\kappa \frac{\omega_d\rho^d}{2r} 
\to\infty,\text{ as }r\to0
\]
contradicting the uniform boundedness of \eqref{eq:uniform-bounded-Step2}.

In consequence, we have proved $\vol(A_r)\ge\frac{\omega_d}{2}(r\rho)^d$, and hence $r\rho\le \big(\frac{2}{\omega_d}\vol(A_r)\big)^{1/d}$. 

By Lemma \ref{lem:nonlocal-isoperimetric}, there exists $C_d>0$ such that
\[
\int_{A_r}\int_{\mathbb{R}^d\setminus A_r}\mathbf{1}_{\{|x-y|<r\rho\}}\,dx\,dy
\ge C_d\,(r\rho)^{d+1}\big(\vol(A_r)\big)^{(d-1)/d}.
\]
Therefore,
\[
\frac{r^{-d-1}}{\vol(A_r)}\int_{A_r}\int_{\mathbb{R}^d\setminus A_r}\mathbf{1}_{\{|x-y|<r\rho\}}\,dxdy \ge  \frac{C_d\,(r\rho)^{d+1}\big(\vol(A_r)\big)^{(d-1)/d}}{r^{d+1}\vol(A_r)}= C_d\, \rho^{d+1}\vol(A_r)^{-1/d}
\]
is bounded from above, for sufficiently small $r>0$. This implies that $\vol(A_r)$ is uniformly bounded from below, for sufficiently small $r>0$. 

\vspace{0.1cm}

To finish Step 2, we note from \eqref{eq:Step2-hypotheses-check} that, the hypotheses of Lemma \ref{lem:Linfty-compact} are satisfied, so there exists a subsequence (denoted by $u_{r_n}$) that converges in $L^1(\Omega)$ to some $u\in L^1(\Omega)$, with $r_n\to0$ as $n\to\infty$. 
Moreover, since each $u_{r_n}$ is $\{0,1\}$-valued, a limit function $u$ takes values in $\{0,1\}$ a.e. (after passing to a further subsequence that converges pointwise a.e.). 
Then, $u=\mathbf1_A$ with $A:=\{x\in\Omega:u(x)=1\}$.  
Note that $\|u_{r_n}\|_1=\vol(A_{r_n})$ is uniformly bounded from below for sufficiently small $r_n>0$. We have $\vol(A)=\|u\|_1=\lim_{n\to\infty}\|u_{r_n}\|_1=\lim_{n\to\infty}\vol(A_{r_n})>0$. 
Consequently, we have proved that $u=\mathbf1_A$ for some $A$ with $\vol(A)>0$. This proves that the asymptotic minimizers $\{A_r\}_{r>0}$ are compact. 

\medskip\noindent
\textbf{Step 3: Liminf inequality $\liminf_{r\to0} r^{-1}h_1(W_r;\Omega)\ge 
C_{1,d}h_1(\Omega)$.
}

Suppose the contrary, that there exist $\delta>0$ and a decreasing sequence $\{r_n\}_{n=1}^\infty$ of positive numbers converging to 0 as $n\to\infty$, such that $r_n^{-1}h_1(W_{r_n};\Omega)< C_{1,d}h_1(\Omega)-\delta$.

Following Step 2, let $\{A_{r_n}\}_{n=1}^\infty\subset\B(\Omega)$ be a sequence of asymptotic minimizers corresponding to $h_1(W_{r_n};\Omega)$. 
By Step 2,  one can extract a subsequence, still denoted by $\{A_{r_n}\}_{n=1}^\infty$, such that $\mathbf 1_{A_{r_n}}$ converges to some $\mathbf1_{A}$ in $L^1$ with $\vol(A)>0$. The Gamma‑convergence liminf inequality implies that, for any  convergent sequence $\mathbf 1_{A_{r_n}}\to \mathbf1_A$ we have
\[
C_{1,d}\|D\mathbf 1_A\|\le \liminf_{n\to\infty} \frac{1}{r_n}\|\mathbf 1_{A_{r_n}}\|_{W_{r_n},1} .
\]
Therefore, we have
\[
C_{1,d}\,\mathcal H^{d-1}(\partial A) \le \liminf_{n\to\infty} \left( r_n^{-1}\, 2\int_{A_{r_n}\times (U\setminus A_r)}W_{r_n}(x,y)\,dx\,dy \right).
\]
Dividing by $\vol(A)=\|\mathbf 1_{A}\|_1 = \lim_{n\to\infty} \|\mathbf 1_{A_{r_n}}\|_1$, we get
\[
C_{1,d}\,\frac{\mathcal H^{d-1}(\partial A)}{\vol(A)} \le \liminf_{n\to\infty} \frac{r_n^{-1}\|\mathbf 1_{A_{r_n}}\|_{W_{r_n},1}}{\|\mathbf 1_{A_{r_n}}\|_1}
= \liminf_{n\to\infty} r_n^{-1}h_1(W_{r_n};\Omega).
\]
Thus, it follows from $h_1(\Omega)\le \frac{\mathcal H^{d-1}(\partial A)}{\vol(A)}$ that
\[
C_{1,d} h_1(\Omega) \le \liminf_{n\to\infty} r_n^{-1}h_1(W_{r_n};\Omega),
\]
which contradicts the assumption presented at the beginning of Step 3. This completes  Step 3. 

Combining Step 3 and Step 1 gives the desired limit:
\[
\lim_{r\to0} r^{-1}h_1(W_r;\Omega)= C_{1,d} h_1(\Omega).
\]
\medskip\noindent
\textbf{Step 4: Convergence of (asymptotic) Cheeger sets.}

Let $A_r$ be a minimizer for $h_1(W_r;\Omega)$ (or an asymptotic minimizer for $h_1(W_r;\Omega)$). By Step 2, $\{\mathbf1_{A_r}\}_{r>0}$ has a subsequence, denoted by $\{\mathbf1_{A_{r_n}}\}_{n=1}^\infty$, that converges in $L^1$ to some $\mathbf1_{A}$ with $A\in\B_\Omega$, $\vol(A)>0$, $\lim_{n\to\infty}r_n=0$, and
\[
C_{1,d}\,\frac{\mathcal H^{d-1}(\partial A)}{\vol(A)} \le \liminf_{n\to\infty} r_n^{-1}h_1(W_{r_n};\Omega)=\lim_{r\to0} r^{-1}h_1(W_r;\Omega)= C_{1,d} h_1(\Omega).
\]
This implies that $A$ is a Cheeger set of $\Omega$. The Hausdorff convergence $A_{r_n}\to A$ follows from the fact that their characteristic functions converge in $L^1$. This completes the proof.
\end{proof}

\subsection{Proof for monotonicity inequalities}
\label{sec:monotonicity}
We shall prove Theorem \ref{thm:monotonicity} which establishes a lot of monotonicities for $\lambda_k(W_{p,q})$ and $\lambda_{\sup}(W_{p,q})$ and thus bridge all those combinatorial 
and geometric 
quantities on graphons in the previous sections. 
For the sake of brevity, we will discuss only the case of graphon, since the signed graphon case, whilst somewhat more cumbersome, is essentially very similar. 

We need the following elementary lemmas.

\begin{lemma}\label{lem:increase-mean-power}
For any probability measure $\mathbb{P}$ on $X$, the function $p\mapsto\big(\int_X |f(x)|^pd\mathbb{P}(x)\big)^{\frac1p}$ is increasing. 
\end{lemma}

\begin{lemma}\cite[Appendix]{Zhang25}\label{lem:elementary-inequality} 
For any $t\ge 1$, $b,a\in\R$, 
$$|b-a|\Big(\frac{|b|^{t}+|a|^{t}}{2}\Big)^{1-\frac 1t}\le \left||b|^{t}\mathrm{sign}(b)-|a|^{t}\mathrm{sign}(a)\right|
\le t|b-a|\Big(\frac{|b|^{t}+|a|^{t}}{2}\Big)^{1-\frac 1t}\,.$$

\end{lemma}

\begin{proof}[Proof of Theorem \ref{thm:monotonicity}]

For generality, we adopt a positive vertex-weight function $\nu:J\to(0,\infty)$, and the standard degree function $\deg$ which is defined as $\deg(x)=\int_JW(x,y)dy$ for $x\in J$; and we use $\|f\|_{p,\nu}:=(\int_J \nu(x)|f(x)|^pdx)^{1/p}$ to denote the weighted $p$-norm of $f$.  
Given $t>1$, denote by $\psi_tf$ a function on $J$ defined via $\psi_tf(x)=|f(x)|^t\mathrm{sign}(f(x))$, $\forall x\in J$; such $\psi_t$ is also called a Mazur map (see e.g. \cite{Benyamini}). 
Since $\frac{p}{q}(1-\frac1t)<1$,  by Lemma \ref{lem:elementary-inequality}, we obtain
\begin{align*}
\|\psi_tf\|_{W,p}^p&=\int_{J\times J} W(x,y)\big||f(x)|^t\mathrm{sign}(f(x))-|f(y)|^t\mathrm{sign}(f(y))\big|^pdxdy
\\&\le t^p\int_{J\times J} W(x,y)|f(x)-f(y)|^p\left(\frac{|f(x)|^t+|f(y)|^t}{2}\right)^{(1-\frac1t)p}dxdy
\\& \le t^p\left(\int_{J\times J} W(x,y)|f(x)-f(y)|^{ps}dxdy\right)^{\frac1s} 
\left(\int_{J\times J} W(x,y)\left(\frac{|f(x)|^t+|f(y)|^t}{2}\right)^{(1-\frac1t)pu}dxdy\right)^{\frac1u}
\\& \le
t^p\|f\|_{W,ps}^{p}\left(\int_{J\times J} W(x,y)\left(\frac{|f(x)|^{tq}+|f(y)|^{tq}}{2}\right)^{\frac{1}{q}(1-\frac1t)pu}dxdy\right)^{\frac1u}
\\& \le
t^p\|f\|_{W,ps}^{p}\left(\int_{J\times J} W(x,y)\frac{|f(x)|^{tq}+|f(y)|^{tq}}{2} dxdy\right)^{\frac{1}{q}(1-\frac1t)p}
\\& =
t^p\|f\|_{W,ps}^{p}\left(\int_{J} \deg(x)\, |f(x)|^{tq} dx\right)^{\frac{1}{q}(1-\frac1t)p}
\\& \le
t^p\|f\|_{W,ps}^{p}\left(C\int_{J} \nu(x) \,|f(x)|^{tq} dx\right)^{\frac{1}{q}(1-\frac1t)p}
\\&=C^{\frac{p}{q}(1-\frac1t)} t^p\|f\|_{W,ps}^{p} \|f^t\|_{q,\nu}^{(1-\frac1t)p}
\end{align*}
where 
$$\frac1s+\frac1u=1,\;\frac{1}{q}(1-\frac1t)pu=1,\; t>1,u>1,\text{ and }C=\mathop{\mathrm{esssup}}\limits_{x\in X} \frac{\deg(x)}{\nu(x)}.$$
Hence, the above inequality implies
$$ \frac{\|\psi_tf\|_{W,p}}{\|\psi_tf\|_q}\le t C^{\frac{1}{q}(1-\frac1t)}\frac{\|f\|_{W,ps}}{\|f\|_{qt}}$$
where we omitted the subscript $\nu$ for simplicity. In particular, when $\nu\equiv 1$, we have $C=\mathop{\mathrm{esssup}}\limits_{x\in X} \deg(x)=\mathop{\mathrm{esssup}}\limits_{x\in X} \int_JW(x,y)dy$ which indicates the (essential) maximum degree of $W$. 

\vspace{0.17cm}

\textbf{Min-max process}: 
For any origin-symmetric subset $A$, let $\psi_r(A):=\{\psi_rf:f\in A\}$ be the image of $A$ under the Mazur map $f\mapsto \psi_rf$. Note that the Mazur map 
$$\begin{array}{l}
  A\to \psi_t(A) \\
  f\mapsto \psi_tf    
\end{array}$$
is an odd homeomorphism and hence  $\mathrm{genus}(\psi_t(A))=\mathrm{genus}(A)$. Thus, 
we can take min-max on both sides of the above inequality to get 
\begin{align*}
\lambda_k(W_{p,q})&=\inf_{\mathrm{genus}(A)\ge k}\sup_{f\in A}\frac{\|f\|_{W,p}}{\|f\|_q} = \inf_{\mathrm{genus}(\psi_t(A))\ge k}\sup_{\psi_tf\in \psi_t(A)}\frac{\|\psi_tf\|_{W,p}}{\|\psi_tf\|_q}
\\& = \inf_{\mathrm{genus}(A)\ge k}\sup_{f\in A}\frac{\|\psi_tf\|_{W,p}}{\|\psi_tf\|_q} 
\le tC^{\frac{1}{q}(1-\frac1t)} \inf_{\mathrm{genus}(A)\ge k}\sup_{f\in A}\frac{\|f\|_{W,ps}}{\|f\|_{qt}}
= tC^{\frac{1}{q}(1-\frac1t)}
\lambda_k(W_{ps,qt}) .
\end{align*}
This implies the inequality on the right-hand side of \eqref{eq:p,q-monotonicity1}. 

The proof of $  \lambda_k(W_{p,q})\ge  2^{1-t}\lambda_k(W_{pt,qt})^t$ also relies on Lemma \ref{lem:elementary-inequality}.
In fact, using the inequality $|b^t-a^t|\ge|b-a|\left(\frac{|a|^t+|b|^t}{2}\right)^{1-\frac1t}$ in Lemma \ref{lem:elementary-inequality}, for $t>1$ and $p\ge 1$, we have
\begin{align*}
\int_{J\times J} |f(x)^t-f(y)^t|^pdxdy&\ge \int_{J\times J} |f(x)-f(y)|^p\left(\frac{|f(x)|^t+|f(y)|^t}{2}\right)^{(1-\frac1t)p}dxdy
\\& \ge \int_{J\times J} |f(x)-f(y)|^p\left(\frac{|f(x)-f(y)|}{2}\right)^{t(1-\frac1t)p}dxdy
\\& =
2^{(1-t)p}\int_{J\times J} |f(x)-f(y)|^{tp}dxdy.
\end{align*}
Thus, 
$$ \frac{\|f^t\|_{W,p}}{\|f^t\|_q}\ge 2^{1-t}\left(\frac{\|f\|_{W,pt}}{\|f\|_{qt}}\right)^t.$$
which yields $  \lambda_k(W_{p,q})\ge  2^{1-t}\lambda_k(W_{pt,qt})^t$ (by a similar \textbf{min-max process} as before). Thus, the inequality on the left-hand side of \eqref{eq:p,q-monotonicity1} holds.

Finally, we focus on 
the monotonicity of the functions $\|W\|_1 ^{-\frac1p}\lambda_k(W_{p,q})$ and  $\mu(J)^{\frac1q}\lambda_k(W_{p,q})$. 

Consider the probability measures $\bar\eta$ and $\bar \mu$ define by
\[\bar\eta(S):=\frac{\int_SW(x,y)dxdy}{\|W\|_1 },\;\forall S\in\B(J\times J),\;\;\text{ and }\;\;\bar \mu(A):=\frac{\mu(A)}{\mu(J)},\;\forall A\in\B(J)\,.\]
It follows from Lemma \ref{lem:increase-mean-power} that the following quantity
$$\frac{\mu(J)^{\frac1q}}{\|W\|_1 ^{\frac1p}}\frac{\|f\|_{W,p}}{\|f\|_q}=\frac{(\int_{J\times J}|f(x)-f(y)|^pd\bar\eta(x,y))^{\frac1p}}{(\int_J|f(x)|^qd\bar\mu(x))^{\frac1q}}$$
is increasing with respect to $p\in[1,\infty]$, and  decreasing with respect to $q\in[1,\infty]$.
Finally, by a similar min-max process as before, one can obtain that $\|W\|_1 ^{-\frac1p}\mu(J)^{\frac1q}\lambda_k(W_{p,q})$ is increasing with respect to $p\in[1,\infty]$, and  decreasing with respect to $q\in[1,\infty]$. 
This concludes the proof.
\end{proof}

\begin{proof}[Proof of Theorem \ref{thm:geo-Sobolev-conve}]
The proof is similar to that of Theorem \ref{th:W_r-M}, and hence we omit the details.
\end{proof}

\begin{proof}[Proof of Theorem \ref{thm:monotonicity-manifold}]
Applying Theorem \ref{thm:monotonicity} to $W_r$, we have 
$$\frac1r\lambda_j((W_r)_{p,q})\le tc_r^{\frac{1}{q}(1-\frac1t)}\frac1r\lambda_j((W_r)_{ps,qt}) $$
where $c_r=\int_{M}W_r(x,y)dy$, and $ps^*=qt^*$ with $s^*$ and $t^*$ being the H\"older conjugates of $s$ and $t$, respectively.  By Theorem \ref{thm:geo-Sobolev-conve}, 
$$\frac1r\lambda_j\big((W_r)_{p,q}\big)\to C_{p,k}^{\frac1p}C_{0,k}^{-\frac1q} \lambda_j(M_{p,q})\;\;\text{ and }\;\;\frac1r\lambda_j((W_r)_{ps,qt}) \to C_{ps,k}^{\frac{1}{ps}}C_{0,k}^{-\frac{1}{qt}} \lambda_j(M_{ps,qt}),\;r\to0^+.$$
By \eqref{eq:int_AW_r=C0k} in the proof of Theorem \ref{thm:pers1-diag}, we have $c_r\to C_{0,k}$ as $r$ tends to 0. 
Combining all the facts above together, we have 
$$C_{p,k}^{\frac1p}C_{0,k}^{-\frac1q} \lambda_j(M_{p,q})\le tC_{0,k}^{\frac{1}{q}(1-\frac1t)}C_{ps,k}^{\frac{1}{ps}}C_{0,k}^{-\frac{1}{qt}} \lambda_j(M_{ps,qt})$$
which concludes the desired inequality in Theorem \ref{thm:monotonicity-manifold}. 

By Theorem \ref{thm:pers1-diag} (or \eqref{eq:AtimesAtovol(A)} in the proof of Theorem \ref{thm:pers1-diag}), $\|W_r\|_1\to C_{0,k}\vol(M)$ as $r\to0$. 
Since $\|W_r\|_1 ^{-\frac1p}\frac1r\lambda_j((W_r)_{p,q})$ is increasing with respect to $p$, we can simply take $r\to0$ to derive that $(C_{0,k}\vol(M))^{-\frac1p}C_{p,k}^{\frac1p}C_{0,k}^{-\frac1q} \lambda_j(M_{p,q})$ is increasing with respect to $p$. 

In a similar way,  $\vol(M)^{\frac1q}C_{p,k}^{\frac1p}C_{0,k}^{-\frac1q} \lambda_j(M_{p,q})$ is decreasing with respect to $q$. 
\end{proof}


\subsection{Graphon $p$-Laplacian}
The Sobolev constants are related to 
$p$-Laplacian equations on graphons, which we shall introduce below. 
This also generalizes the Dirichlet problem for the nonlocal 1-Laplacian on a Euclidean domain \cite{Mazon16}. 
Consider
\begin{equation}\label{eq:nonlocal-1-lap-graphon}
\begin{cases}
\displaystyle \int_{U} W(x,y)\, z(x,y)\, dy = 0, & x \in \Omega,\\[1.2ex]
u(x) = \psi(x), & x \in U \setminus \Omega,
\end{cases}
\end{equation}
where $\psi \in L^{\infty}(U\setminus \Omega)\cap L^{1}(U\setminus \Omega)$, $U$ is an open neighborhood of a closed domain $\Omega\subset\R^d$, 
and $W(x,y)$ is a given nonnegative, bounded, integrable function that is symmetric
(i.e., $W(y,x)=W(x,y)$ for all $(x,y)\in U\times U$)
and has a uniform positive lower bound near the diagonal $\{(x,x):x\in U\}$.
Problem \eqref{eq:nonlocal-1-lap-graphon} seeks $u \in L^{1}(\Omega)$ such that there exists
$z \in L^{\infty}(U\times U)$ with $\|z\|_{\infty}\le 1$ and
$z(x,y)\in \operatorname{Sgn}(u(x)-u(y))$ and $z(y,x)=-z(x,y)$, where 
$$\operatorname{Sgn}(t):=\begin{cases}
\{1\},&\text{ if }t>0,\\
[-1,1],&\text{ if }t=0,\\
\{-1\},&\text{ if }t<0.\\
\end{cases}$$
The energy functional $\mathcal{F}:L^1(U)\to\R$  corresponding to \eqref{eq:nonlocal-1-lap-graphon} is defined by
\begin{equation}\label{eq:graphon-tv}
\mathcal{F}(u) := \int_{U}\int_{U} W(x,y)\, |u(y)-u(x)|\, dx\,dy,
\end{equation}
where we always require $u(x)=\psi(x)$ for $x\in U\setminus \Omega$.

\begin{theorem}
$u$ is a solution of \eqref{eq:nonlocal-1-lap-graphon} if and only if it is a minimiser of the functional \eqref{eq:graphon-tv}.
\end{theorem}

\begin{proof}
Suppose that $u$ is a solution of \eqref{eq:nonlocal-1-lap-graphon}. From 
$z(x,y)(u(x)-u(y)) = |u(x)-u(y)|$, $z(y,x)=-z(x,y)$ and $|z(x,y)|\le1$, we obtain
\begin{align*}
0 &= \int_{U}\int_{U} W(x,y) z(x,y)dy(w(x)-u(x)) dx = \int_{U\times U} \frac{W(x,y) z(x,y)\big(w(x)-w(y)-(u(x)-u(y))\big)}{2} dydx \\
&\le \frac12 \int_{U}\int_{U} W(x,y)\, |w(x)-w(y)|\, dy\,dx
 - \frac12 \int_{U}\int_{U} W(x,y)\, |u(x)-u(y)|\, dy\,dx = \frac{\mathcal{F}(w) - \mathcal{F}(u)}{2},
\end{align*}
for every $w\in L^1(U)$. Hence $u$ is a minimiser of $\mathcal{F}$.

Conversely, if $u$ is a minimiser of $\mathcal{F}$, then it is a critical point,
so $0 \in \partial_{u}\mathcal{F}(u)$, which is equivalent to \eqref{eq:nonlocal-1-lap-graphon}.
\end{proof}

\begin{prop}\label{pro:poincare-ineq}
There exists $C>0$ such that for any $u\in L^p(U)$ with $u=\psi$ on $U\setminus\Omega$, there holds
\[
\int_{\Omega} |u(x)|^{p}\, dx \le C \int_{\Omega}\int_{U} W(x,y)\, |u(x)-u(y)|^{p}\, dy\,dx
 + C \int_{U\setminus \Omega} |\psi(y)|^{p}\, dy.
\]
\end{prop}

\begin{proof}
According to the condition on $W$ given after  \eqref{eq:nonlocal-1-lap-graphon},  there exist $r,\alpha>0$ such that $W(x,y)\ge \alpha$ whenever $|x-y|<r$. 
Set
\[
B_0 = \{ x \notin \Omega : d(x,\Omega) < \tfrac{r}{2} \}, \qquad
B_1 = \{ x \in \Omega : d(x,\Omega^{c}) < \tfrac{r}{2} \},
\]
and inductively
\[
B_i = \left\{ x \in \Omega \,\Big\backslash\bigcup_{j=1}^{i-1} B_j : d(x,B_{i-1}) < \tfrac{r}{2} \right\}, \quad i=2,3,\dots
\]
Since $\Omega$ is bounded, it can be covered by finitely many $B_i$, say $\Omega = \bigcup_{i=1}^{m} B_i$.
Hence
\begin{align*}
&\int_{\Omega}\int_{U} W(x,y)\, |u(x)-u(y)|^{p}\, dy\,dx
 \\ \ge& \int_{B_j}\int_{B_{j-1}} W(x,y)\, |u(x)-u(y)|^{p}\, dy\,dx \\
 \ge &\int_{B_j}\int_{B_{j-1}} W(x,y) \left( \frac{1}{2^{p-1}} |u(x)|^{p} - |u(y)|^{p} \right) dy\,dx \\
 =\,& \frac{1}{2^{p-1}} \int_{B_j} \left( \int_{B_{j-1}} W(x,y)\, dy \right) |u(x)|^{p}\, dx  - \int_{B_{j-1}} \left( \int_{B_j} W(x,y)\, dx \right) |u(y)|^{p}\, dy \\
 \ge\,& \alpha_j \int_{B_j} |u(x)|^{p}\, dx - C \int_{B_{j-1}} |u(y)|^{p}\, dy,
\end{align*}
for $j=1,2,\dots,m$.
Therefore, $\int_{B_1} |u|^p$ is controlled by a linear combination of
$\int_{\Omega}\int_{U} W(x,y) |u(x)-u(y)|^pdxdy$ and $\int_{B_0} |u|^p$. 
Similarly, $\int_{B_2} |u|^p$ is controlled by a linear combination of the same energy and
$\int_{B_1} |u|^p$, and recursively we obtain that $\int_{\Omega} |u|^p$ is controlled
by a linear combination of $\int_{\Omega}\int_{U} W(x,y) |u(x)-u(y)|^pdxdy$ and
$\int_{U\setminus \Omega} |u(y)|^pdy = \int_{U\setminus \Omega} |\psi(y)|^pdy$. 
This concludes the proof.
\end{proof}

\begin{theorem}
Problem \eqref{eq:nonlocal-1-lap-graphon} admits a solution.
\end{theorem}

\begin{proof}
Consider the functional
\begin{equation}\label{eq:graphon-p-tv}
\mathcal{F}_p(u) := \int_{U}\int_{U} W(x,y)\, |u(y)-u(x)|^{p}\, dx\,dy,
\end{equation}
with $p>1$. Let $\{u_n\}$ be a minimizing sequence.
By Proposition \ref{pro:poincare-ineq}, $\{u_n\}$ is bounded in $L^p$, 
so there exists a weakly convergent subsequence (still denoted by $u_n$)
with $u_n \rightharpoonup u$. By weak lower semicontinuity, $u$ is a minimizer of $\mathcal{F}_p$. 
Thus, for any $w \in L^p(U)$ with $\operatorname{supp} w \subset \Omega$,
\begin{align*}
0 &= \left. \frac{d}{dt} \right|_{t=0} \mathcal{F}_p(u + t w) = \int_{U}\int_{U} W(x,y)\, |u(y)-u(x)|^{p-2} (u(y)-u(x)) (w(y)-w(x))\, dx\,dy \\
&= 2 \int_{U}\int_{U} W(x,y)\, |u(y)-u(x)|^{p-2} (u(y)-u(x))\, w(y)\, dx\,dy.
\end{align*}
Since $w|_{\Omega^c}=0$, we have
\[
\int_{\Omega} \left( \int_{U} W(x,y)\, |u(y)-u(x)|^{p-2} (u(y)-u(x))\, dx \right) w(y)\, dy = 0,
\]
and this yields
\[
\int_{U} W(x,y)\, |u(y)-u(x)|^{p-2} (u(y)-u(x))\, dx = 0, \quad \text{ for a.e. } y \in \Omega.
\]
Conversely, from the above equality, one obtains
\[
\int_{U}\int_{U} W(x,y)\, |u(y)-u(x)|^{p-2} (u(y)-u(x)) (w(y)-w(x))\, dx\,dy = 0.
\]
Now set $w(x) = \max\{u(x) - \|\psi\|_{\infty},0\}$. Then
\begin{align*}
0 &= \int_{U}\int_{U} W(x,y)\, |u(y)-u(x)|^{p-2} (u(y)-u(x)) (w(y)-w(x))\, dx\,dy 
\\&\ge \int_{U}\int_{U} W(x,y)\, |w(y)-w(x)|^{p}\, dx\,dy,
\end{align*}
implying that $w$ is a constant, which we denote by $C$; if $C>0$, then
$u(x) = C + \|\psi\|_{\infty}$ for all $x \in U$,
contradicting $u=\psi$ on $U\setminus \Omega$.
Hence, $C\le 0$, and then $u(x) \le \|\psi\|_{\infty}$. Similarly, one has 
$u(x) \ge -\|\psi\|_{\infty}$, so $\|u\|_{\infty} \le \|\psi\|_{\infty}$.

Thus, for every $p>1$, all such minimisers $u_p$ are uniformly bounded in $L^\infty$,
so they have a weakly convergent subsequence in $L^p$; denote the weak limit by $u$. Finally, it is straightforward to verify that $u$ is indeed a solution of \eqref{eq:nonlocal-1-lap-graphon}.
\end{proof}

\subsection{Analysis on uniform hyper-graphons}
Hypergraph limits have also attracted a great deal of interest \cite{Backhausz25}. In this section, we provide a result on hyper-graphon. 

Given a hyper-graphon $W:\overbrace{J\times J\times \cdots\times J}^{k\text{ times}}\to [0,1]$ which is measurable and symmetric,  for $p\ge 1$ and $k\in\mathbb{Z}_+$, we introduce the $p$-Rayleigh quotient 
$$\Phi_{p}(f)=\frac{\int_{J^k}W(x_1,\cdots,x_k)|f(x_1)+\cdots+f(x_k)|^pdx_1\cdots dx_k}{\int_J|f(x)|^pdx}$$
and the adjacency-tensor-energy ratio
$$T(f)=\frac{\int_{J^k}W(x_1,\cdots,x_k)f(x_1)\cdots f(x_k)dx_1\cdots dx_k}{\int_J|f(x)|^kdx}.$$
For $i=1,2,\cdots$, let $$c_i(\Phi_p)=\inf_{S\subset L^\infty(J):\mathrm{genus}(S)\ge i}\sup_{f\in S}\Phi_p(f)~\text{ and }~c_i(T)=\inf_{S\subset L^\infty(J):\mathrm{genus}(S)\ge i}\sup_{f\in S}T(f)$$ be the $i$-th min-max critical values of $\Phi_p(\cdot)$ and $T(\cdot)$, respectively. We similarly define
$$c_{i}^+(\Phi_p)=\sup_{S\subset L^\infty(J):\mathrm{genus}(S)\ge i}\inf_{f\in S}\Phi_p(f)~\text{ and }~c_i^+(T)=\sup_{S\subset L^\infty(J):\mathrm{genus}(S)\ge i}\inf_{f\in S}T(f).$$
The notions of $c_i^+(\Phi_p)$ and $c_i^+(T)$ are suitable for bounded hypergraphon, but it does not make sense to generalize them directly to the $p$-Rayleigh quotient on Riemannian manifolds. 
\begin{theorem}\label{thm:hypergraphon}
For $i=1,2,\cdots$, we have $c_i(\Phi_p)\ge k^pc_i(T)$ and $c_i^+(\Phi_p)\ge k^pc_i^+(T)$. 
\end{theorem}

\begin{proof}
We need the following elementary argument: For any $a_1,\cdots,a_k\in[0,+\infty)$, 
$$p\mapsto \frac{(a_1^{\frac kp}+\cdots +a_k^{\frac kp})^p}{k^p}$$
decreasingly converges to $a_1\cdots a_k$ as $p$ goes to $\infty$;  while if there exist $i,j$ with $a_ia_j<0$, then 
$$p\mapsto \frac{|a_1^{\frac kp}+\cdots +a_k^{\frac kp}|^p}{k^p}$$ converges to 0  as $p$ diverges to $\infty$. 

By the argument above, we first have the following limit property
\begin{align*}
\lim\limits_{p\to+\infty}\frac{\Phi_p(f^{\frac kp})}{k^p}&=\frac{\int_{J^k}W(x_1,\cdots,x_k)|f(x_1)\cdots f(x_k)|1_{f(x_1)>0,\cdots,f(x_k)>0}dx_1\cdots dx_k}{\int_J|f(x)|^kdx}    \\&\ge 
\frac{\int_{J^k}W(x_1,\cdots,x_k)f(x_1)\cdots f(x_k)dx_1\cdots dx_k}{\int_J|f(x)|^kdx}=T(f).
\end{align*}
Furthermore, we have 
\begin{align*}
\frac{\Phi_p(f^{\frac kp})}{k^p}&\ge \frac{\int_{J^k}W(x_1,\cdots,x_k)\big(f^{\frac kp}(x_1)+\cdots+f^{\frac kp}(x_k)\big)^p1_{f(x_1)>0,\cdots,f(x_k)>0}dx_1\cdots dx_k}{k^p\int_J|f(x)|^kdx}
\end{align*}
in which the right hand side decreasingly converges to 
$$\frac{\int_{J^k}W(x_1,\cdots,x_k)|f(x_1)\cdots f(x_k)|1_{f(x_1)>0,\cdots,f(x_k)>0}dx_1\cdots dx_k}{\int_J|f(x)|^kdx} $$
which is an upper bound of $T(f)$. Thus, $\frac{\Phi_p(f^{\frac kp})}{k^p}\ge T(f)$, i.e., $\Phi_p(f^{\frac kp})\ge k^pT(f)$. Therefore, using the min-max process as shown in the proof of Theorem \ref{thm:monotonicity}, we obtain the desired conclusion of Theorem \ref{thm:hypergraphon}. 
\end{proof}

\begin{remark}\label{rem:tensor-like}
By Theorem \ref{thm:hypergraphon},  for $p=2$ and $k\ge 3$, the spectrum of the bounded symmetric linear operator corresponding to the Rayleigh quotient $\Phi_2$, gives an upper bound for the tensor-like eigenvalue associated to $T$. 

Moreover, Theorems \ref{thm:monotonicity} and \ref{thm:hypergraphon} imply that when $p\ne 2$ and $k=2$, the spectrum of the adjacency operator associated with graphon $W$ provides lower bounds for min-max eigenvalues of the nonlinear graphon $p$-Laplacian. 
In particular, if  $p$ is an even number and $k=2$,  the spectrum of  the adjacency operator associated with $W$ provides lower bounds for special tensor-like eigenvalues. This is because for even number $p$, the eigenproblem of graphon $p$-Laplacian can be reformulated as a tensor-like eigenproblem. 
\end{remark}

\bibliography{biblimit}  
\bibliographystyle{amsplain}
\appendix

\section{Appendix}


\subsection{Cut Distance for Graphons}\label{app:cut-distance}

{The cut norm} was first introduced by Frieze and Kannan \cite{FriezeKannan99}. Following Lov\'asz's notion, let $W: [0,1]^2 \to \mathbb{R}$ be a bounded, measurable kernel. The {cut norm} of $W$ is defined as
\[
\|W\|_{\square}: = \sup_{S, T \subseteq [0,1]} \left| \int_{S \times T} W(x,y) \, dx \, dy \right|,
\]
where the supremum is taken over all measurable subsets $S, T \subseteq [0,1]$.

Let $\mathcal{S}_{[0,1]}$ denote the group of all measure-preserving bijections $\varphi: [0,1] \to [0,1]$. For a kernel $W$, define its pullback under $\varphi$ as
\[
W^{\varphi}(x,y): = W(\varphi(x), \varphi(y)).
\]
The cut distance between two graphons $U$ and $W$ is defined as
\[
\delta_{\square}(U, W): = \inf_{\varphi \in \mathcal{S}_{[0,1]}} \|U - W^{\varphi}\|_{\square}= \inf_{\varphi, \psi \in \mathcal{S}_{[0,1]}} \|U^{\varphi} - W^{\psi}\|_{\square}.
\]
For a finite graph $G$ with $n$ vertices, one associates a stepgraphon $W_G$ on $[0,1]^2$ by partitioning $[0,1]$ into $n$ intervals of measure $1/n$, and setting $W_G(x,y) = 1$ if the corresponding vertices are adjacent, and $0$ otherwise. The cut distance between two graphs $G$ and $H$ is then defined as the cut distance between their corresponding graphons, i.e., 
$\delta_{\square}(G, H) := \delta_{\square}(W_G, W_H)$.

The cut distance $\delta_{\square}$ is a \emph{pseudometric} on the space of all graphons. If $\delta_{\square}(U, W) = 0$, then $U$ and $W$ are weakly isomorphic (i.e., they define the same limit of dense graph sequences). Convergence in this metric is equivalent to convergence of all subgraph densities (by the Lov\'asz--Szegedy theorem \cite{Lovasz-Szegedy}). 
To obtain a metric that separates equivalence classes, we take the quotient of the space of graphons by measure-preserving transformations. The cut distance on the equivalence classes $[U]$ and $[W]$ is given by
\[
\delta_{\square}([U], [W]) = \inf_{\varphi, \psi \in \mathcal{S}_{[0,1]}} \|U^{\varphi} - W^{\psi}\|_{\square}.
\]

\subsection{$TL^p$ Convergence 
}
\label{Appsec:TLp}

The notion of $TL^p$ convergence was introduced by Garc\'ia Trillos and Slep\v{c}ev (see \cite{Garcia-Slepcev, GarciaSlepcev18}) to provide a rigorous framework for comparing functions defined on different spaces—typically discrete point clouds and a continuous domain—when studying the continuum limit of variational problems. It combines measure convergence with a transport-based notion of function convergence, and it is particularly well suited for proving results on Gamma-convergence. 


Let \(X \subset \mathbb{R}^d\) be a bounded open set and let \(\mu\) be a Borel probability measure on \(X\). Consider a sequence of discrete sets
\[
X_n = \{x_1^n, \dots, x_{k_n}^n\} \subset X,
\]
equipped with the empirical measures
\[
\mu_n := \frac{1}{k_n} \sum_{i=1}^{k_n} \delta_{x_i^n}.
\]

\begin{defn}[$TL^p$ convergence]
A sequence of functions \(u_n \in L^p(X_n, \mu_n)\) is said to \emph{converge in the \(TL^p\) topology} to a function \(u \in L^p(X, \mu)\) if the following 
conditions hold:

\begin{enumerate}
\item \textbf{Measure convergence:} 
\(\mu_n \stackrel{*}{\rightharpoonup} \mu\) weakly, i.e. for every continuous bounded function \(\varphi\) on \(X\),
\[
\int_X \varphi \,\mathrm{d}\mu_n \;\longrightarrow\; \int_X \varphi \,\mathrm{d}\mu.
\]

\item \textbf{Function convergence via transport:} 
There exists a sequence of \emph{stagnating transport maps} \(T_n: X \to X_n\) such that
\begin{itemize}
\item each \(T_n\) is Borel measurable and pushes forward \(\mu\) to \(\mu_n\), i.e.
\[
(T_n)_{\#}\mu = \mu_n,
\]
which means that for every Borel set \(A \subset X_n\), \(\mu(T_n^{-1}(A)) = \mu_n(A)\);
\item the maps are \emph{stagnating} in the sense that
\[
\|T_n - \mathrm{id}\|_{L^p(X,\mu)} \;\longrightarrow\; 0;
\]
\item and the pulled-back functions converge in \(L^p(X,\mu)\):
\[
\|u_n \circ T_n - u\|_{L^p(X,\mu)} \;\longrightarrow\; 0.
\]
\end{itemize}
\end{enumerate}
\end{defn}

The \(TL^p\) topology was specifically designed to prove Gamma-convergence of discrete energy functionals (e.g., graph-based total variation, \(p\)-Dirichlet energies, or nonlocal perimeters) to their continuum limits. Under suitable assumptions on the point clouds (e.g., uniform distribution with respect to \(\mu\)), one can show that the discrete energies \emph{Gamma-converge} in the \(TL^p\) sense to a continuous energy, thereby ensuring that minimizers (and critical values) of the discrete problems converge to those of the continuum problem.

\subsection{A slight generalization of Gamma convergence}\label{sec:generalized-Gamma}


Based on the works of Trillos and Slep\v{c}ev \cite{Garcia-Slepcev,W8L8}, and Degiovanni and Marzocchi \cite{DM2014}, we shall present some basics on 
the general theory of (sequential) Gamma-convergence. Let $X$ be a nonempty set, $(X_n)_{n \ge 1}$ be a sequence of nonempty sets, and suppose that an appropriate notion of convergence of sequences from $(X_n)_{n \ge 1}$ to $X$ has been defined, i.e., the meaning of the convergence $\lim_{n\to+\infty}x_n=x$ is clear for suitable $x_n\in X_n$ and $x\in X$. For example, suppose $X_n\subset X$, $\forall n$, and suppose that $X$ is a topological space, then the sequential convergence $x_n\to x$ is reasonable. 
Let $F_n: X_n \to [0, \infty]$ be a sequence of functionals.

\begin{defn}[Gamma-convergence]\label{def:gamma-converge}
The sequence of functionals $\{F_n\}_{n \in \mathbb{N}}$ is said to \emph{Gamma-converge} to the functional $F: X \to [0, \infty]$ as $n \to \infty$ if the following two inequalities hold:
\begin{enumerate}
\item \textbf{Liminf inequality:} for every $x \in X$ and every sequence $\{x_n\}_{n \in \mathbb{N}}$ with $x_n \in X_n$ converging to $x$, we have
\[
\liminf_{n \to \infty} F_n(x_n) \ge F(x).
\]
\item \textbf{Limsup inequality:} for every $x \in X$ there exists a sequence $\{x_n\}_{n \in \mathbb{N}}$ with $x_n \in X_n$ converging to $x$ such that
\[
\limsup_{n \to \infty} F_n(x_n) \le F(x).
\]
\end{enumerate}
We call $F$ the \emph{Gamma-limit} of the sequence $\{F_n\}$.
\end{defn}

When $X_n = X$ for all $n \in \mathbb{N}^+$, the Gamma-convergence in Definition \ref{def:gamma-converge} reduces to the standard Gamma-convergence.

\begin{prop}\label{pro:inf_Xn<infX}
Suppose that a sequence of functions $\{F_n: X_n \to \mathbb{R}\}$ Gamma-converges to $F: X \to \mathbb{R}$. Then
\[
\limsup_{n \to +\infty} \inf_{x \in X_n} F_n(x) \le \inf_{x \in X} F(x).
\]
If, moreover, there exists a sequence $\{x_n\}$ that has a cluster point (i.e., limit point) in $X$ and satisfies
\begin{equation}\label{eq:aymptotic-min}
\lim_{n \to +\infty} \bigl( F_n(x_n) - \inf_{x \in X_n} F_n(x) \bigr) = 0,    
\end{equation}
then $F$ attains its minimum, $F_n(x_n) \to \min_X F$, and every limit point of $\{x_n\}$ is a minimizer of $F$ on $X$.
\end{prop}

\begin{proof}
The Liminf and Limsup inequalities together imply that for every $x \in X$ there exists a sequence $\{x_n\}$ converging to $x$, with $x_n \in X_n$, such that $\lim_{n \to +\infty} F_n(x_n) = F(x)$. 
Hence
\[
F(x) = \lim_{n \to +\infty} F_n(x_n) \ge \limsup_{n \to +\infty} \inf_{x \in X_n} F_n(x),
\]
which implies $\limsup_{n \to +\infty} \inf_{X_n} F_n \le \inf_X F$.

If, moreover, there exists a sequence $\{x_n\}$ with a convergent subsequence (still denoted by $\{x_n\}$) converging to some $x \in X$, the liminf inequality gives $\liminf_{n \to +\infty} F_n(x_n) \ge F(x)$. Combining this with \eqref{eq:aymptotic-min} gives 
\[
F(x) \ge \limsup_{n \to +\infty} \inf_{x \in X_n} F_n(x)\ge \liminf_{n \to +\infty} \inf_{x \in X_n} F_n(x)=\liminf_{n \to +\infty}F_n(x_n) \ge F(x),
\]
which yields 
$$\lim_{n \to +\infty}  F_n(x_n)=\lim_{n \to +\infty} \inf_{x \in X_n} F_n(x)=F(x)=\inf_{x'\in X}F(x')=\min_{x'\in X}F(x'). $$
These conclude the proof.
\end{proof}

Without compactness assumptions, one cannot conclude that $\inf_{X_n} F_n(x) \to \inf_X F(x)$ in general.

\begin{defn}[Asymptotic equicoercivity]
A sequence of functionals $\{F_n: X_n \to \mathbb{R}\}$ is called \emph{asymptotically equicoercive} if for every subsequence $\{u_{n_k} \in X_{n_k}\}$ satisfying
\[
\sup_{k \in \mathbb{N}^+} F_{n_k}(u_{n_k}) < +\infty,
\]
the sequence $(u_{n_k})_{k \ge 1}$ possesses a convergent subsequence.
\end{defn}

\begin{prop}\label{pro:equicoercive}
Suppose that an asymptotically equicoercive sequence of functions $\{F_n: X_n \to \mathbb{R}\}$ Gamma-converges to $F: X \to \mathbb{R}$. Then $F$ attains its minimum, the asymptotic minima of $\{F_n\}$ converge to the minimum of $F$, and every limit point of asymptotic minimizers (i.e., $x_n$ used in \eqref{eq:aymptotic-min}) is a minimizer of $F$.
\end{prop}

\begin{proof}
Take $u_n$ such that
\[
\lim_{n \to +\infty} \bigl( F_n(u_n) - \inf_{x \in X_n} F_n(x) \bigr) = 0.
\]
By Proposition \ref{pro:inf_Xn<infX},
\[
\limsup_{n \to +\infty} F_n(u_n) = \limsup_{n \to +\infty} \inf_{x \in X_n} F_n(x) \le \inf_X F(x)<+\infty.
\]
Thus $\sup_{n \in \mathbb{N}^+} F_n(u_n) < +\infty$. By asymptotic equicoercivity, $(u_n)$ has a convergent subsequence. Applying Proposition \ref{pro:inf_Xn<infX} again completes the proof.  
\end{proof}

Let $X_n$ be a metric space, $\mathcal{K}(X_n)$ the topological space of compact subsets of $X_n$ (equipped with the Hausdorff metric), and define $\mathcal{F}_n(K) = \sup_{x \in K} F_n(x)$ for all $K \in \mathcal{K}(X_n)$.

\begin{prop}[see Proposition 2.5 in \cite{DM2014}]\label{pro:mathcalF_nvsF_n}
Consider a sequence of functions $\{F_n: X_n \to \mathbb{R}\}$ and the associated functional sequence $\{\mathcal{F}_n: \mathcal{K}(X_n) \to \mathbb{R}\}$. Then, $(F_n)_{n \ge 1}$ is asymptotically equicoercive if and only if $(\mathcal{F}_n)_{n \ge 1}$ is asymptotically equicoercive.
\end{prop}


Now, continuing from \cite{DM2014}, we present some results on minimax critical values.

Let $\mathcal{K}$ and $\mathcal{K}_n$ be the families of compact subsets of the unit spheres of the normed spaces $X$ and $X_n$, respectively, where $\dim X_n=n$. Let $\mathcal{K}_s^{(m)} \subset \mathcal{K}$  be the family of symmetric compact subsets of $X$, whose Krasnoselskii genus is at least $m$. 
Similarly, let  $\mathcal{K}_{s,n}^{(m)}=\{K\in \mathcal{K}_n:-K=K,\,\mathrm{genus}(K)\ge m\}$. Given continuous even functions $F:X\to\R$ and $F_n:X_n\to\R$, 
the minimax critical values $(c_m)$ and $(c_{m,n})$ are defined respectively by
\begin{equation}
c_m = \inf_{K \in \mathcal{K}_s^{(m)}} \max_{u \in K} F(u),
\qquad
c_{m,n} = \inf_{K \in \mathcal{K}_{s,n}^{(m)}} \max_{u \in K} F_n(u).
\end{equation}

Let $\mathcal{F}: \mathcal{K} \to \mathbb{R}$ be defined by $\mathcal{F}(K) = \max_{u \in K} F(u)$. %
Let $\mathcal{F}_n: \mathcal{K}_n \to \mathbb{R}$ and $\mathcal{F}_n^{(m)}: \mathcal{K}_{s,n}^{(m)} \to \mathbb{R}$ be defined analogously. 
The corresponding Hausdorff metric $d_{\mathcal{K}}$ on $\mathcal{K}$ is given by
\[
d_{\mathcal{K}}(K_1, K_2) = \max \left\{ \max_{u \in K_1} d_X(u, K_2),\ \max_{u \in K_2} d_X(u, K_1) \right\}.
\]

\begin{prop}\label{pro:tv-continuous}
$\mathcal{F}$ is continuous on $(\mathcal{K}, d_{\mathcal{K}})$, and $\mathcal{F}_n$ is continuous on the compact space $(\mathcal{K}_n, d_{\mathcal{K}_n})$.
\end{prop}

\begin{proof}
We only need to show that $\mathcal{F}(K_n) \to \mathcal{F}(K)$ whenever $d_{\mathcal{K}}(K_n, K) \to 0$. Observe that there exists a subsequence $\{K_{n_k}\} \subset \{K_n\}$ such that
\[
\lim_{k} \mathcal{F}(K_{n_k}) = \limsup_{n} \mathcal{F}(K_n).
\]
Since $K_n$ is compact, we may assume that $u_n \in K_n$ satisfies $F(u_n) = \max_{u \in K_n} F(u)$. Since $K$ is compact, there exists $v_n \in K$ with $d(u_n, v_n) \to 0$. Thus $\{v_{n_k}\}$ has a convergent subsequence; without loss of generality, assume $v_{n_k} \to v^*$. Then $d(u_{n_k}, v^*) \to 0$, so
\[
\limsup_{n} \mathcal{F}(K_n) = \lim_{k} F(u_{n_k}) = F(v^*) \le \mathcal{F}(K).
\]
Conversely, let $u_0 \in K$ satisfy $F(u_0) = \mathcal{F}(K)$. There exists $u_n \in K_n$ such that $d(u_0, u_n) \to 0$. Hence $F(u_n) \to F(u_0)$. Therefore,
\[
\liminf_{n} \mathcal{F}(K_n) \ge \lim_{n} F(u_n) = F(u_0) = \mathcal{F}(K).
\]
The above proof also applies to $\mathcal{F}_n$. Finally, let $S(X_n)$ denote the unit sphere of $X_n$. 
Since $S(X_n)$ is a closed subset of a finite-dimensional space, it is compact; hence $(\mathcal{K}_n, d_{\mathcal{K}_n})$, being the space of compact subsets of the compact set $S(X_n)$, equipped with the Hausdorff metric, is compact.  
\end{proof}

For the next result, we need the following concept of Morse critical points \cite{D2010,Ioffe96}.

\begin{defn}[{\cite{Ioffe96}}]
Let $f: X \to \mathbb{R}$ be continuous. A point $u \in X$ is called \emph{Morse regular} if there exist an open neighbourhood $U$ of $u$ and a continuous map $\eta: U \times [0,1] \to X$ 
such that $\eta(x,0) = x$ for all $x \in U,$ and
\[
f(\eta(x,t)) < f(x) \quad \text{for all } (x,t) \in U \times (0,1].
\]
A point is called a \emph{Morse critical point} if it is not Morse regular. A real number $c$ is called a \emph{Morse regular value} if no Morse critical point takes the value $c$; otherwise it is called a \emph{Morse critical value}.
\end{defn}

Note that if $f$ is a Morse function, then the usual critical points coincide with Morse critical points. However, Morse critical points are defined for much more general continuous functions. In general (even for smooth functions), the concept of Morse critical points is strictly stronger than the usual critical points. For example, $f:\R\to\R,\; x\mapsto  x^3$ is a smooth non‑Morse function, and $0$ is a (usual) critical point of $f$ but not a Morse critical point.

\begin{theorem}Suppose that an asymptotically equicoercive sequence $\{F_n: X_n \to \mathbb{R}\}_{n \ge 1}$ Gamma-converges to $F$. Then,  
for every $m$, we have $c_{m,n} \to c_m$ as $n \to \infty$. Moreover, there exists $K^* \in \mathcal{K}_s^{(m)}$ such that $\max\limits_{u \in K^*} F(u) = c_m$; and for such $K^*$, there exists $u^* \in K^*$ such that $F(u^*) = c_m$ and $u^*$ is a Morse critical point of $F$.
\end{theorem}

\begin{proof}
By Proposition \ref{pro:tv-continuous}, $\mathcal{F}_n$ is continuous on the compact space $(\mathcal{K}_n, d_{\mathcal{K}_n})$. Since $\mathcal{K}_{s,n}^{(m)}$ is closed in $\mathcal{K}_n$ (hence compact), $\mathcal{F}_n$ attains its infimum on $\mathcal{K}_{s,n}^{(m)}$, i.e., there exists $K_n^* \in \mathcal{K}_{s,n}^{(m)}$ such that $\mathcal{F}_n(K_n^*) = \inf_{K \in \mathcal{K}_{s,n}^{(m)}} \mathcal{F}_n(K)$.

By Proposition \ref{pro:mathcalF_nvsF_n} (or the relevant statement), $\{\mathcal{F}_n\}_{n \ge 1}$ is also asymptotically equicoercive. 
From~\cite[Corollary 4.4]{DM2014}, $\mathcal{F}_n^{(m)}$ Gamma-converges to $\mathcal{F}^{(m)}$. By Proposition \ref{pro:equicoercive}, $c_{m,n} \to c_m$ and $\mathcal{F}^{(m)}$ reaches its minimum; denote a minimizer by $K^*$. Finally, we shall prove that there exists $u^* \in \arg \max_{u \in K^*} f(u)$ such that $u^*$ is a Morse critical point of $f$.

Suppose, for contradiction, that every point in $K_1 := \{x \in K^* : f(x) = \max_{y \in K} f(y)\}$ is not a Morse critical point. Applying the result of \cite{D2010} to the compact set $K_1$, for any open set $U \supset K_1$, there exists an open set $V$ with $K_1 \subset V \subset \overline{V} \subset U$ and an odd flow $\eta: X \times [0,1] \to X$ such that $\eta(-x,t)=-\eta(x,t)$, 
\[
\eta(x,t) = x, \quad \forall x \in X \setminus V,\; t \in [0,1],
\]
and
\[
f(\eta(x,t)) < f(x), \quad \forall x \in V,\; t > 0.
\]
Therefore,
\[
\max_{y \in \eta(K^*,1)} f(y) < \max_{x \in K^*} f(x).
\]
However, the property of Krasnoselskii genus implies $\mathrm{genus}(\eta(K^*,1))\ge \mathrm{genus}(K^*)\ge m$, which leads to  
$c_m\le \max_{y \in \eta(K^*,1)} f(y)<\max_{x \in K^*} f(x)=c_m$, a contradiction. 
In consequence, there exists $u^* \in K^*$ such that $F(u^*) = c_m$ and $u^*$ is a Morse critical point of $F$.  
\end{proof}



 \end{document}